\documentclass[12pt]{article}
\usepackage[pdfpagemode=UseNone,pdfstartview=FitH,hypertex]{hyperref}
\usepackage{amsfonts}
\usepackage{amssymb}
\usepackage{mathrsfs}
\usepackage{srcltx}
\usepackage{amsmath}
\usepackage{amsthm}
\usepackage{amstext}
\usepackage{amsopn}
\usepackage{graphicx}
\newtheorem{theorem}{Theorem}[section]
\newtheorem{lemma}[theorem]{Lemma}
\newtheorem{proposition}[theorem]{Proposition}
\newtheorem{corollary}[theorem]{Corollary}

\theoremstyle{definition}
\newtheorem{definition}[theorem]{Definition}

\theoremstyle{remark}
\newtheorem{remark}[theorem]{Remark}

\numberwithin{equation}{section}

\newcommand{\N}{\mathbb{N}}
\newcommand{\Z}{\mathbb{Z}}
\newcommand{\R}{\mathbb{R}}

\newcommand{\ba}{\begin{array}}
\newcommand{\ea}{\end{array}}

\newcommand{\di}{\displaystyle}

\newcommand{\ds}{\displaystyle}

\newcommand{\dint}{\displaystyle \int}
\newcommand{\id}[1][]{id_{\, #1}}

\begin{document}
\title{Ground and bound state solutions for a Schr\"odinger system with linear and nonlinear couplings in $\mathbb{R}^N$}
\author{Kanishka Perera$^{a}$, Cyril Tintarev$^{b}$, Jun
Wang$^{c}$, Zhitao Zhang$^{d}$\footnote{Corresponding author. E-mail\ addresses: kperera@fit.edu (K.\! Perera),\ tintarev@math.uu.se (C.\! Tintarev),\ wangmath2011@126.com (J.\! Wang),\ zzt@math.ac.cn (Z.\! Zhang)}
\\
 {\small $^{a}$Department of Mathematical Sciences, Florida Institute of Technology,}\\
 {\small Melbourne, FL 32901, USA}\\
{\small $^{b}$Department of Mathematics, Uppsala University,75 106 Uppsala, Sweden}\\
{\small $^{c}$Faculty of Science, Jiangsu University, Zhenjiang, Jiangsu, 212013, P.R. China}\\
{\small $^{d}$Academy of Mathematics and Systems Science, Chinese}\\
{\small Academy of Sciences, Beijing 100190, P.R. China}}

\date{}
\maketitle

\begin{quote}
\noindent {\bf Abstract:} We study the
existence of ground and bound state solutions for a system of coupled
Schr\"odinger equations with linear and nonlinear couplings in $\mathbb{R}^N$. By studying the limit system and using concentration compactness arguments, we prove the existence of ground and bound state solutions under suitable assumptions. Our results are new even for the limit system.
\\


\end{quote}

\section{Introduction and main results}

\setcounter{section}{1} \setcounter{equation}{0}

In this paper we study the $2$-component coupled nonlinear Schr\"odinger system
\begin{equation}\label{auto}
\begin{cases}
- \Delta
u+u=(a_0(x)+a(x))|u|^{p-2}u+(\beta_0+\beta(x))|u|^{\frac{p}{2}-2}u|v|^{\frac{p}{2}}
+(\kappa_0+\kappa(x))v,\\
- \Delta
v+v=(b_0(x)+b(x))|v|^{p-2}v+(\beta_0+\beta(x))|u|^{\frac{p}{2}}
|v|^{\frac{p}{2}-2}v+(\kappa_0+\kappa(x))u,\\
(u,v) \in H^1(\mathbb{R}^N) \times H^1(\mathbb{R}^N),
\end{cases}
\end{equation}
where $2<p<\frac{2N}{N-2}$ if $N\geq3$, and $2<p<\infty$ if $N=1,2$.
We assume that the following condition holds:
\begin{itemize}
  \item [$(A_0)$]$a_0,\, b_0 \in L^\infty(\mathbb{R}^N)$ are positive $\Z^N$-periodic
functions, $\beta_0, \kappa_0\in\mathbb{R}$, $a, b, \beta, \kappa
\in L^\infty(\mathbb{R}^N)$ go to zero as $|x| \to \infty$, and
\begin{equation}\label{a-2}
\begin{split}
&\inf_{x \in \mathbb{R}^N}\, (a_0(x) + a(x))>0,\quad \inf_{x \in \mathbb{R}^N}\, (b_0(x) + b(x))>0,\\
&0<\kappa_0+\inf_{x \in \mathbb{R}^N}\, \kappa(x)\leq
\kappa_0+\sup_{x \in \mathbb{R}^N}\, \kappa(x)<1.
\end{split}
\end{equation}
\end{itemize}
This system of equations is related to the following important
Schr\"{o}dinger system with linear and nonlinear couplings arising
in Bose-Einstein condensates (see \cite{DKNF}):
\begin{equation}\label{original}
\left\{ \begin{aligned}
 &-i\frac{\partial \Phi}{\partial t}=\Delta\Phi-V(x)\Phi+\mu_1|\Phi|^2\Phi+\beta|\Psi|^2\Phi+\kappa\Psi,~t>0,x\in \Omega,  \\
 &-i\frac{\partial \Psi}{\partial t}=\Delta\Psi-V(x)\Psi+\mu_2|\Psi|^2\Psi+\beta|\Phi|^2\Psi+\kappa\Phi,~t>0,x\in \Omega,
\end{aligned} \right.
\end{equation}
where $\Omega$ is a smooth domain in $\mathbb{R}^N$, $V$ is the
relevant potential, typically consisting of a magnetic trap and/or
an optical lattice, and $\Phi$ and $\Psi$ are the (complex-valued) condensate
wave functions. The intra- and interspecies
interactions are characterized by the coefficients $\mu_1, \mu_2 > 0$ and $\beta$, respectively, while $\kappa$ denotes the strength of the radio-frequency (or electric-field) coupling. This system
also arises in the study of fiber optics, where the solution $(\Phi,\Psi)$ is two coupled electric-field envelopes of the same wavelength, but of different polarizations, and the linear coupling
is generated either by a twist applied to the fiber in the case of two linear polarizations, or by an elliptic deformation of the fiber's core in the case of circular polarizations.
Looking for solitary wave solutions of the form
$\Phi(x,t)=e^{i\lambda t}u(x),\, \Psi(x,t)=e^{i\lambda t}v(x)$, where $\lambda > 0$ is a constant, leads to the following elliptic system for
$u$ and $v$:
\begin{equation}\label{li-zhang1.1}
\left\{ \begin{aligned}
 &-\Delta u+(\lambda +V(x))u=\mu_1 u^3+\beta u v^2+\kappa v  & ~~in  &~~~~ \Omega, \\
 &-\Delta v+(\lambda+V(x)) v=\mu_2 v^3+\beta u^2 v+\kappa u  & ~~in  & ~~~~ \Omega,\\
 &u =v=0~~\mbox{on}~~\partial \Omega~ (\hbox{or}~ u,v\in H^1(\mathbb{R}^N)~ \hbox{if}~
 \Omega=\mathbb{R}^N),
\end{aligned} \right.
\end{equation}
which has received considerable attention in recent years.

Interesting existence and multiplicity results in various domains
for system \eqref{li-zhang1.1} with $\kappa = 0$ have been obtained
in
\cite{AC,ACR,BDW,Dancer-Wei-2009-TRMS,DW2,DWZ1,DWZ2,DWZ3,DWZ4,NTTV,TT,TV}.
In particular, bifurcation results were obtained in \cite{BDW}, and
larger systems and their limiting equations were considered in
\cite{DWZ1,DWZ2,DWZ3,DWZ4}. It was shown in Ambrosetti-Colorado
\cite{AC} that when $\kappa = 0$ and $V(x) \equiv 0$ there exist
constants $\beta_2 > \beta_1 > 0$ such that this sytem has a
positive ground state solution for $\beta > \beta_2$ and no positive
ground state solution for $\beta < \beta_1$. A solution is called a
ground state solution (or positive ground state solution) if its
energy is minimal among all the nontrivial solutions (or all the
positive solutions) of \eqref{auto} or \eqref{li-zhang1.1}.

Existence and asymptotic behavior of multi-bump solitons of system \eqref{li-zhang1.1} with $\kappa\neq 0,\beta = 0$ and $V(x) \equiv 0$ were studied in Ambrosetti-Cerami-Ruiz \cite{ACR} using perturbation methods. Ambrosetti-Cerami-Ruiz \cite{ACR-1} studied positive ground and bound state solutions of the system
\begin{equation}\label{ACR1}
\left\{ \begin{aligned}
 &-\Delta u+ u =(1+a(x)) |u|^{p-1}u+\kappa v, \\
 &-\Delta v+ v=(1+b(x)) |v|^{p-1}v+\kappa u,\\
 &(u,v) \in H^1(\mathbb{R}^N) \times H^1(\mathbb{R}^N)
\end{aligned} \right.
\end{equation}
when $\kappa>0$, $N\geq 2, 1<p<2^*-1,$ where $$ 2^*=\left\{\begin{array}{lll}\frac{2N}{N-2}, ~\hbox{if}~ N\geq 3,\\
+\infty,~\hbox{if}~ N=2.\end{array}\right.$$
Under the general assumptions $a(x), b(x)\in L^{\infty}(\mathbb{R}^N)$, $\lim\limits_{|x|\rightarrow+\infty} a(x)=\lim\limits_{|x|\rightarrow+\infty} b(x)=0$, $\inf\limits_{x\in\mathbb{R}^N}(1+a(x))>0,\, \inf\limits_{x\in\mathbb{R}^N}(1+b(x))>0$, and additional suitable hypotheses, they used concentration compactness type arguments to obtain some interesting results about the existence of positive
ground and bound state solutions.
\par
System \eqref{li-zhang1.1} with both $\kappa \ne 0$ and $\beta \ne
0$ has been much less studied. Topological methods were used in
Beitia-Garca-Torres \cite{BPT} to obtain a positive bound state
solution, and variational methods and index theory were used in
Li-Zhang \cite{Li-Kui-Zhangzhitao-2015-Preprint} to obtain a ground
state solution and infinitely many positive bound state solutions.
Some existence results when $V(x) \equiv 0$ were obtained in
Tian-Zhang \cite{tian-zhang} using variational and bifurcation
arguments. In the present paper we consider the system \eqref{auto},
which generalizes \eqref{ACR1}, with $\kappa_0 \ne 0$ and $\beta_0
\ne 0$. Our results seem to be new even for the limit system
\eqref{auto} when $a(x)=b(x)=\beta(x)=\kappa(x)\equiv0$.\par

Our first results is Theorem \ref{th1.1} below, which is concerned with the existence of a positive ground state solution of \eqref{auto}.

\begin{theorem}\label{th1.1}
Suppose that $(A_0)$ holds, $0 < \kappa_0 < 1$, and
\begin{equation}\label{a-6}
\kappa(x) \ge 0, \quad a(x) \ge 0, \quad b(x) \ge 0, \quad \beta(x) \ge 0,
\end{equation}
with at least one of the inequalities in \eqref{a-6} strict on a set of positive measure. Then \eqref{auto} has a ground state solution. For the case $u=v$, if
\begin{equation}\label{a-7}
\kappa(x)\ge 0, \quad a(x)+b(x)+2\beta(x)\ge 0,
\end{equation}
with at least one of these inequalities strict on a set of positive measure, then \eqref{auto} has a ground state solution.
\end{theorem}

\begin{remark}
We give conditions more general than \eqref{a-6} and \eqref{a-7} that guarantee the existence of a ground state solution of \eqref{auto} (see Theorem \ref{zh0127} in Section 6), which are not stated here in order to simplify notation.
\end{remark}

Next we study \eqref{auto} when the functions $a, b,
\kappa$ and $\beta$ are non-positive. In this case there exists no
ground state solution (see Lemma \ref{Lemma 7.1}), and bound states
must be sought at higher levels. We
assume that $a_0$ and $b_0$ are positive constants and $p=4$ in
\eqref{auto}. Then we have the following result.
\begin{theorem}\label{th1.2}
Let $a_0, b_0>0$, $0 < \kappa_0 < 1$, $p=4$, and assume that
\eqref{a-2} and one of the following conditions holds:
\begin{itemize}
  \item [(1)] $\beta_0\geq3$;
  \item [(2)] $1\leq\beta_0\leq3$ and $w(0)\leq
\sqrt{\frac{2\kappa_0(1+\beta_0)}{(3-\beta_0)(1-\kappa_0)}}$, where
$w$ is the unique positive solution of the scalar equation $-\Delta
u+u=u^{3},\ u\in H^1(\mathbb{R}^{N})$;
  \item [(3)] $-1<\beta_0<1$, $w(0)\leq
\sqrt{\frac{2\kappa_0(1+\beta_0)}{(3-\beta_0)(1-\kappa_0)}}$, and
\begin{itemize}
\item if $N=1$, then $\kappa_0$ or $\beta_0-1$ is sufficiently small,
\item if $N=2,3$, then $\beta_0, \kappa_0>0$ are sufficiently small, or $|\beta_0|$ is sufficiently small and $\kappa_0$ is close to 1.
\end{itemize}
\end{itemize}
If $\kappa(x), a(x), b(x), \beta(x) \le 0$, with at least one of the
inequalities strict on a set of positive measure, then system
\eqref{auto} has a positive bound state solution provided that
\begin{equation}
R_0:=\left(1 + \frac{|\kappa|_\infty}{1 - \kappa_0}\right)^2\left(1
- \max
\left\{\frac{|a|_\infty}{a_0},\frac{|b|_\infty}{b_0},\frac{|\beta|_\infty}{\beta_0}\right\}\right)^{-1}
\end{equation}
is sufficiently small.
\end{theorem}

\begin{remark}
\begin{itemize}
\item [(1)] We give a more precise assumption on $R_0$ (see Lemma \ref{Lemma 7.5}), which is not stated here in order to simplify notation.
\item [(2)] A key step in the proof of Theorem \ref{th1.2} is showing the uniqueness and nondegeneracy of the positive solution of the limit system
\begin{equation}\label{a-8}
\begin{cases}
- \Delta
u+u=a_0|u|^{p-2}u+\beta_0|u|^{\frac{p}{2}-2}u|v|^{\frac{p}{2}}
+\kappa_0v,\\
- \Delta v+v=b_0|v|^{p-2}v+\beta_0|u|^{\frac{p}{2}}
|v|^{\frac{p}{2}-2}v+\kappa_0u,\\
(u,v) \in H^1(\mathbb{R}^N) \times H^1(\mathbb{R}^N).
\end{cases}
\end{equation}
Compared to the paper \cite{ACR-1}, we have more coupled terms
here, namely $|u|^{\frac{p}{2}} |v|^{\frac{p}{2}-2}v$ and
$|u|^{\frac{p}{2}-2}u|v|^{\frac{p}{2}}$. These terms present new
difficulties for proving the uniqueness and nondegeneracy of the positive
solution of \eqref{a-8} for general $p$. Using some ideas from
\cite{Dancer-Wei-2009-TRMS,WeiYao2012-CPAA,Ikoma-2009-Nodea}, we prove this for the case $p=4$ here (see Lemma \ref{lem-3.5}). The proof for general $p$ is an interesting open problem.
\end{itemize}
\end{remark}

\begin{remark}
The main results are Theorems \ref{th1.1} and \ref{th1.2}. For the limit system \eqref{c-1} used to prove the main results, Lemmas \ref{lem-3.1} -- \ref{lem-3.5} are new and of independent interest.
\end{remark}

The paper is organized as follows. In Section 2, we give some
notations and preliminaries. In section 3, we study the existence
and asymptotic behavior of the positive solution of the limit system
\eqref{c-1}, and consider the uniqueness and nondegeneracy of
the positive solution of \eqref{c-1} when $p=4$ and $a_0, b_0$
are positive constants. In section 4, we give a concentration
compactness result. In section 5, we study existence of ground state
solutions for a functional-analytic model of our problem. In section
6, we give several sufficient conditions for the existence of ground state solutions.
In section 7, we prove the existence of bound state solutions of
system \eqref{auto} when $p = 4$ and $a_0, b_0$ are positive
constants.
\par

\section{Preliminaries}

\setcounter{section}{2} \setcounter{equation}{0}

We will use the following notations:
\begin{itemize}
\item for a positive function or constant $M$, $\|\cdot\|_M$ is the equivalent norm on $H^1(\mathbb{R}^N)$ defined by $\di\|u\|_{M}^{2}=\int_{\mathbb{R}^{N}}(|\nabla u|^{2}+M|u|^{2})$;
 \item $\|(u,v)\|_E=(\|u\|^{2}+\|v\|^{2})^{1/2}$ is the norm of $E=H^1(\mathbb{R}^{N})\times H^1(\mathbb{R}^N)$, where $\|\cdot\|$
 is a norm on $H^1(\mathbb{R}^N)$;
\item for $1\le p<\infty$, $|\cdot|_p$ is the usual norm of $L^p(\mathbb{R}^N)$ defined by $|u|_p=\ds\left(\int_{\mathbb{R}^N}|u|^p\right)^{1/p}$;
\item $2^*$ is the critical Sobolev exponent given by $2^{*}=\frac{2N}{N-2}$ if $N\geq3$, and $2^{*}=\infty$ if $N=1,\, 2$;
\item $C_i,\, i=1,2,\dots$ denote positive constants.
\end{itemize}

The energy functional associated with the system \eqref{auto} is given by
\begin{equation}\label{b-1}
\begin{split}
\Phi(u,v)&=\frac{1}{2}(\|u\|^{2}+\|v\|^2)-\frac{1}{p}\int_{\mathbb{R}^{N}}\left[(a_0(x)+a(x))|u|^{p}
+(b_0(x)+b(x))|v|^{p}\right]\\
&\quad-\frac{2}{p}\int_{\mathbb{R}^{N}}(\beta_0+\beta(x))|u|^{\frac{p}{2}}|v|^{\frac{p}{2}}-\int_{\mathbb{R}^{N}}
(\kappa_0+\kappa(x))uv,~ \forall(u,v) \in H^1(\mathbb{R}^N) \times H^1(\mathbb{R}^N).
\end{split}
\end{equation}
To obtain nontrivial solutions of \eqref{auto}, we use the associated Nehari manifold
\begin{equation}\label{b-2}
\mathscr{N}=\{z=(u,v)\in E\setminus\{(0,0)\}: \Phi'(u,v)(u,v)=0\}.
\end{equation}
Clearly, $\Phi\in C^{2}(E,\mathbb{R})$ and all nontrivial
critical points of $\Phi$ are on $\mathscr{N}$. For $(u,v)\in\mathscr{N}$,
\begin{equation}\label{b-3}
\begin{split}
\|u\|^{2}+\|v\|^2&=\int_{\mathbb{R}^{N}}\left[(a_0(x)+a(x))|u|^{p}
+(b_0(x)+b(x))|v|^{p}\right]\\
&\quad+2\int_{\mathbb{R}^{N}}(\beta_0+\beta(x))|u|^{\frac{p}{2}}|v|^{\frac{p}{2}}+2\int_{\mathbb{R}^{N}}
(\kappa_0+\kappa(x))uv\\
&\leq
C\, (\|u\|^p+\|v\|^p)+(\kappa_0+\sup\kappa(x))(\|u\|^2+\|v\|^2)
\end{split}
\end{equation}
by the Sobolev inequality. Since $\kappa_0+\sup\kappa(x)<1$ and $p > 2$, it follows from this that
$\|u\|+\|v\|\geq\sigma>0$ for all $(u,v)\in\mathscr{N}$, so $\mathscr{N}$ is uniformly bounded away from the origin in $E$.

Set
\[
c=\inf_{(u,v)\in\mathscr{N}}\, \Phi(u,v).
\]
A pair of functions $(u,v)\in\mathscr{N}$ such that $\Phi(u,v)=c$ will be called
a ground state solution of \eqref{auto}. We have
\begin{equation}\label{b-5}
\begin{split}
\Phi|_{\mathscr{N}}(u,v)&=\left(\frac{1}{2}-\frac{1}{p}\right)\bigg[\int_{\mathbb{R}^{N}}\left[(a_0(x)+a(x))|u|^{p}
+(b_0(x)+b(x))|v|^{p}\right]\\
&\quad+2\int_{\mathbb{R}^{N}}(\beta_0+\beta(x))|u|^{\frac{p}{2}}|v|^{\frac{p}{2}}\bigg]\\
&=\left(\frac{1}{2}-\frac{1}{p}\right)\left[\|u\|^{2}+\|v\|^{2}-2\int_{\mathbb{R}^{N}}(\kappa_0+\kappa(x))uv\right]\\
&\geq\left(\frac{1}{2}-\frac{1}{p}\right)\left[\|u\|^{2}+\|v\|^{2}-2\, (\kappa_0+\sup \kappa(x))\|u\|\|v\|\right]\\
&\geq\left(\frac{1}{2}-\frac{1}{p}\right)[1-(\kappa_0+\sup \kappa(x))]\left[\|u\|^{2}+\|v\|^{2}\right].
\end{split}
\end{equation}
Since $\kappa_0+\sup \kappa(x)<1$, it follows that $c>0$. First we
have the following lemma regarding the role of $c$.

\begin{lemma}\label{lem-2.1}
If $c$ is attained at  $z\in\mathscr{N}$, then $z$ is a solution of
\eqref{auto}.
\end{lemma}

\begin{proof}
Assume that $z_{0}=(u_{0},v_{0})\in\mathscr{N}$ is such that
$\Phi(u_{0},v_{0})=c$. According to \cite[Theorem
4.1.1]{ChangKungChing2005},  $\mathscr{N}$ is a locally
differentiable manifold and so there exists a Lagrange multiplier
$\ell\in\mathbb{R}$ such that
\begin{equation}\label{b-6}
\Phi'(u_{0},v_{0})=\ell G'(u_{0},v_{0}),
\end{equation}
where $G(u,v)=\Phi'(u,v)(u,v)$. We infer from
$z_{0}=(u_{0},v_{0})\in\mathscr{N}$ and $\kappa_0+\sup \kappa(x)<1$
that
\begin{equation}\label{b-7}
\begin{split}
G'(u_{0},v_{0})(u_{0},v_{0})&=2(\|u_{0}\|^{2}+\|v_{0}\|^{2})
-p\int_{\mathbb{R}^{N}}\left[(a_0(x)+a(x))|u_0|^{p}
+(b_0(x)+b(x))|v_0|^{p}\right]\\
&\quad-2p\int_{\mathbb{R}^{N}}(\beta_0+\beta(x))|u_0|^{\frac{p}{2}}|v_0|^{\frac{p}{2}}-4\int_{\mathbb{R}^{N}}
(\kappa_0+\kappa(x))u_0v_0\\
&=(2-p)\left[(\|u_0\|^2+\|v_0\|^2)-2\int_{\mathbb{R}^{N}}
(\kappa_0+\kappa(x))u_0v_0\right]<0.
\end{split}
\end{equation}
Testing the equation \eqref{b-6} with $(u_{0},v_{0})$, it follows from \eqref{b-7} that $\ell=0$. Thus, we
have $\Phi'(u_{0},v_{0})=0$, i.e., $z_{0}$ is a critical point of
$\Phi$.
\end{proof}

\section{The limit equations}

\setcounter{section}{3} \setcounter{equation}{0}

\subsection{Ground state solution}

In this subsection we study the existence and asymptotic
behavior of the positive solution of the limit system
\begin{equation}\label{c-1}
\begin{cases}
- \Delta
u+u=a_0(x)|u|^{p-2}u+\beta_0|u|^{\frac{p}{2}-2}u|v|^{\frac{p}{2}}
+\kappa_0v,\quad x\in\mathbb{R}^{N},\\
- \Delta v+v=b_0(x)|v|^{p-2}v+\beta_0|u|^{\frac{p}{2}}
|v|^{\frac{p}{2}-2}v+\kappa_0u,\quad x\in\mathbb{R}^{N},
\end{cases}
\end{equation}
where $2<p<2^*$, $a_0(x)$ and $b_0(x)$ are $1$-periodic positive
functions. The energy functional corresponding to \eqref{c-1} is
defined by
\begin{equation}\label{c-2}
\begin{split}
\Phi_0(u,v)&=\frac{1}{2}(\|u\|^{2}+\|v\|^2)-\frac{1}{p}\int_{\mathbb{R}^{N}}\left[a_0(x)|u|^{p}
+b_0(x)|v|^{p}\right]\\
&\quad-\frac{2}{p}\int_{\mathbb{R}^{N}}\beta_0|u|^{\frac{p}{2}}|v|^{\frac{p}{2}}-\int_{\mathbb{R}^{N}}
\kappa_0uv.
\end{split}
\end{equation}
The corresponding Nehari manifold is
\begin{equation}\label{c-3}
\mathscr{N}_0=\{z=(u,v)\in E\setminus\{(0,0)\}:
\Phi'_0(u,v)(u,v)=0\}.
\end{equation}
Clearly, $\Phi_0\in C^{2}(E,\mathbb{R})$ and all nontrivial
solutions are contained in $\mathscr{N}_0$. Set
\begin{equation}\label{c-4}
c_0=\inf_{(u,v)\in\mathscr{N}_0}\, \Phi_0(u,v).
\end{equation}
As in \eqref{b-3}-\eqref{b-5}, one can show that if $0<\kappa_0<1$, $\mathscr{N}_0$ is uniformly bounded away from the origin $(0,0)$. Moreover, if we replace $\mathscr{N}$ and $c$ by $\mathscr{N}_0$ and $c_0$, respectively, the conclusion of Lemma \ref{lem-2.1} remains true for $0<\kappa_0<1$.

Now we are ready to prove the existence of a ground state solution of
\eqref{c-1}.
\begin{lemma}\label{lem-3.1}
If $0<\kappa_0<1$, the periodic system \eqref{c-1} has a positive
ground state solution $(u,v)\in\mathscr{N}_0$.
\end{lemma}

\begin{proof}
Let $w_0$ denote the positive solution of $-\Delta
u+u=a_0(x)|u|^{p-2}u, u\in H^{1}(\mathbb{R}^{N})$. Then
$(w_0,0)\in\mathscr{N}_0$, and $\mathscr{N}_0\neq\emptyset$. Let
$\{(u_{n},v_{n})\}\subset\mathscr{N}_0$ be a minimizing sequences.
By using the Ekeland's variational principle type arguments(see
\cite[Lemma 3.10]{wangjun2014pre} or \cite{Willem1996book}), we can
assume that there exists a subsequence of
$\{(u_{n},v_{n})\}\subset\mathscr{N}_0$(still denote by
$(u_{n},v_{n})$) such that
\begin{equation}\label{c-5}
\Phi_0(u_{n},v_{n})\rightarrow c_{0},\quad
\Phi_0'|_{\mathscr{N}_0}(u_{n},v_{n})\rightarrow0.
\end{equation}
Similar to \eqref{b-3} and \eqref{b-5}, it follows from
$\kappa_0+\sup \kappa(x)<1$ that
$0<\sigma\leq\|u_{n}\|+\|v_{n}\|\leq C_1$. We claim that
$\Phi_0'(u_{n},v_{n})\rightarrow0$ as $n\rightarrow\infty$. Indeed,
it is clear that
\begin{equation}\label{c-6}
o(1)=\Phi_0|_{\mathscr{N}_0}'(u_{n},v_{n})=\Phi_0'(u_{n},v_{n})-\ell_{n}G'_{0}(u_{n},v_{n}),
\end{equation}
where $\ell_{n}\in\mathbb{R}$ and $G_0(u,v)=\Phi_0'(u,v)(u,v)$. As
in Lemma \ref{lem-2.1}, one can check that
$G_0(u_n,v_n)(u_n,v_n)\leq-C_2<0$. So, we know that $\ell_{n} \to 0$ in
\eqref{c-6}. Thus, it follows that
\begin{equation}\label{c-7}
\Phi_0(u_{n},v_{n})\rightarrow c_{0},\quad
\Phi_0'(u_{n},v_{n})\rightarrow0.
\end{equation}
From the boundedness of $\{(u_n,v_n)\}$, without loss of generality
we assume that $u_{n}\rightharpoonup u_{0}$, $v_{n}\rightharpoonup
v_{0}$ in $H^{1}(\mathbb{R}^{N})$, $u_{n}\rightarrow u_{0}$ and
$v_{n}\rightarrow v_{0}$ in $L_{loc}^{p}(\mathbb{R}^{N})$, $\forall
p\in(2,2^{*})$.

We claim that $\{(u_{n},v_{n})\}$ is nonvanishing, i.e., there
exists $R>0$ such that
\begin{equation}\label{c-8}
\liminf_{n\rightarrow\infty}\int_{B_{R}(y_{n})}(u_{n}^{2}+v_{n}^{2})\geq\delta>0,
\end{equation}
where $y_{n}\in\mathbb{R}^{N}$ and
$B_{R}(y_{n})=\{y\in\mathbb{R}^{N}: |y-y_{n}|\leq R\}$. Arguing by
contradiction, if \eqref{c-8} is not satisfied, then
$\{(u_{n},v_{n})\}$ is vanishing, i.e.,
\begin{equation}\label{c-9}
\lim_{n\rightarrow\infty}\sup_{y\in\mathbb{R}^{N}}\int_{B_{r}(y)}(u_{n}^{2}+v_{n}^{2})=0,\
\text{for\ all}\ r>0.
\end{equation}
According to Lions's concentration compactness lemma(see \cite[Lemma
1.21]{Willem1996book}) that $u_{n}\rightarrow0$ and
$v_{n}\rightarrow0$ in $L^{t}(\mathbb{R}^{N})(\forall
t\in(2,2^{*}))$. So, we infer from
$\Phi_0'(u_{n},v_{n})(u_{n},v_{n})=0$ that
\begin{equation}\label{c-10}
\begin{split}
\|u_n\|^{2}+\|v_n\|^2&=\int_{\mathbb{R}^{N}}\left[a_0(x)|u_n|^{p}
+b_0(x)|v_n|^{p}\right]\\
&\quad+2\int_{\mathbb{R}^{N}}\beta_0|u_n|^{\frac{p}{2}}|v_n|^{\frac{p}{2}}+2\int_{\mathbb{R}^{N}}
\kappa_0u_nv_n\rightarrow0,
\end{split}
\end{equation}
as $n\rightarrow\infty$. This contradicts with
$\|u_n\|+\|v_n\|\geq\sigma>0$. Hence, \eqref{c-8} holds. Moreover,
there exist $\{k_n\}\subset\mathbb{Z}^{N}$ and $R_0>R>0$ such that
\begin{equation}\label{c-11}
\liminf_{n\rightarrow\infty}\int_{B_{R_0}(k_{n})}(u_{n}^{2}+v_{n}^{2})\geq\frac{\delta}{2}>0.
\end{equation}

Set $\tilde{u}_{n}=u_{n}(x+k_{n})$ and
$\tilde{v}_{n}=v_{n}(x+k_{n})$. Since $a_0(x)$ and $b_0(x)$ are
$1$-periodic functions, it follows that the norms and $\Phi$ are
invariance under the translations $x\mapsto x+k_n$. Thus, we can
assume that $\tilde{u}_{n}\rightharpoonup \tilde{u}_{0}$,
$\tilde{v}_{n}\rightharpoonup \tilde{v}_{0}$ in
$H^{1}(\mathbb{R}^{N})$, $\tilde{u}_{n}\rightarrow \tilde{u}_{0}$
and $\tilde{v}_{n}\rightarrow \tilde{v}_{0}$ in
$L_{loc}^{t}(\mathbb{R}^{N})(\forall t\in(2,2^{*}))$. Moreover, it
follows from \eqref{c-11} that
\begin{equation}\label{c-12}
\liminf_{n\rightarrow\infty}\int_{B_{R_0}(0)}(\tilde{u}_{n}^{2}+\tilde{v}_{n}^{2})\geq\frac{\sigma}{2}>0.
\end{equation}
So, we have $\tilde{u}_{0}\neq0$ or $\tilde{v}_{0}\neq0$.
Furthermore, it follows from the weak continuous of $\Phi_0'$ and
\eqref{c-7} that $\Phi_0'(\tilde{u}_{0},\tilde{v}_{0})=0$ and
$\tilde{z}_{0}=(\tilde{u}_{0},\tilde{v}_{0})\in\mathscr{N}$. As in
\cite{Li-Kui-Zhangzhitao-2015-Preprint}, we define the following
inner product
\begin{equation}\label{c-13}
\left<(u_1,v_1),(u_2,v_2)\right>=\left<(u_1,v_1),(u_2,v_2)\right>_{1}-
\kappa_0\int_{\mathbb{R}^{N}}\left(u_1v_2+u_2v_1\right),
\end{equation}
where $\left<(\cdot,\cdot),(\cdot,\cdot)\right>_{1}$ denotes the
inner product in $E$. Correspondingly, the induced norm denotes by
$\|(\cdot,\cdot)\|_{\kappa_0}$. Furthermore, it follows from
$0<\kappa_0<1$ that the norms $\|(\cdot,\cdot)\|_{\kappa_0}$ and
$\|(\cdot,\cdot)\|_{E}$ are equivalent in $E$. Hence, we infer from
the weak lower semicontinuity of the norm that
\begin{equation}\label{c-14}
\begin{split}
c_0&\leq\Phi_0(\tilde{u}_{0},\tilde{v}_{0})=\left(\frac{1}{2}-\frac{1}{p}\right)\left[\|\tilde{u}_{0}\|^{2}+\|\tilde{v}_{0}\|^{2}
-2\int_{\mathbb{R}^{N}}\kappa_0\tilde{u}_{0}\tilde{v}_{0}\right]\\
&=\left(\frac{1}{2}-\frac{1}{p}\right)\|(\tilde{u}_{0},\tilde{v}_{0})\|_{\kappa_0}^2
\leq\liminf_{n\rightarrow\infty}\left(\frac{1}{2}-\frac{1}{p}\right)\|(\tilde{u}_{n},\tilde{v}_{n})\|_{\kappa_0}^2\\
&=\liminf_{n\rightarrow\infty}\left(\frac{1}{2}-\frac{1}{p}\right)\left[\|\tilde{u}_{n}\|^{2}+\|\tilde{v}_{n}\|^{2}
-2\int_{\mathbb{R}^{N}}\kappa_0\tilde{u}_{n}\tilde{v}_{n}\right]\\
&=\liminf_{n\rightarrow\infty}\Phi_0(\tilde{u}_{n},\tilde{v}_{n})=
\lim_{n\rightarrow\infty}\Phi_0(u_{n},v_{n})=c_0.
\end{split}
\end{equation}
So, $\tilde{z}_{0}=(\tilde{u}_{0},\tilde{v}_{0})\neq(0,0)$ is a
ground state solution of \eqref{c-1}. Finally, we prove that
$\tilde{z}_{0}=(\tilde{u}_{0},\tilde{v}_{0})$ is positive.
Obviously, there exists unique $t>0$ such that
$t(|\tilde{u}_{0}|,|\tilde{v}_{0}|)\in\mathscr{N}_0$. So, one sees
that
\begin{equation}\label{c-15}
\begin{split}
t^{p-2}&=\frac{\|\tilde{u}_{0}\|^{2}+\|\tilde{v}_{0}\|^{2}
-2\int_{\mathbb{R}^{N}}\kappa_0|\tilde{u}_{0}||\tilde{v}_{0}|}{\int_{\mathbb{R}^{N}}\left[a_0(x)|u|^{p}
+b_0(x)|v|^{p}+2\beta_0|u|^{\frac{p}{2}}|v|^{\frac{p}{2}}\right]}\\
&\leq\frac{\|\tilde{u}_{0}\|^{2}+\|\tilde{v}_{0}\|^{2}
-2\int_{\mathbb{R}^{N}}\kappa_0\tilde{u}_{0}\tilde{v}_{0}}{\int_{\mathbb{R}^{N}}\left[a_0(x)|u|^{p}
+b_0(x)|v|^{p}+2\beta_0|u|^{\frac{p}{2}}|v|^{\frac{p}{2}}\right]}=1.
\end{split}
\end{equation}
Furthermore, we infer from
$t(|\tilde{u}_{0}|,|\tilde{v}_{0}|)\in\mathscr{N}_0$ that
\begin{equation}\label{c-16}
\begin{split}
c_0\leq\Phi_0(t|\tilde{u}_{0}|,t|\tilde{v}_{0}|)=t^2\left(\frac{1}{2}-\frac{1}{p}\right)\|(|\tilde{u}_{0}|,|\tilde{v}_{0}|)\|_{\kappa_0}^2
\leq\left(\frac{1}{2}-\frac{1}{p}\right)\|(\tilde{u}_{0},\tilde{v}_{0})\|_{\kappa_0}^2=c_0.
\end{split}
\end{equation}
Thus, one deduces that $t=1$,
$(|\tilde{u}_{0}|,|\tilde{v}_{0}|)\in\mathscr{N}_0$ and
$\Phi(|\tilde{u}_{0}|,|\tilde{v}_{0}|)=c_0$. Hence, we can assume
that $(\tilde{u}_{0},\tilde{v}_{0})$ is a nonnegative solution of
\eqref{c-1}. Moreover, by using the maximum principle we know that
$(\tilde{u}_{0},\tilde{v}_{0})$ is a positive ground state solution of \eqref{c-1}.
\end{proof}

Next we consider the asymptotic behavior of the ground state
solution of \eqref{c-1} as the parameter $\kappa_0\rightarrow0$ in the simple case where $p=4$, and $a_0$ and $b_0$ are constants. To emphasize the dependency on $\kappa_0$, in the following we write $c_0^{\kappa_0}$, $\mathscr{N}_0^{\kappa_0}$ and $\Phi_0^{\kappa_0}$ for $c_0$, $\mathscr{N}_0$ and $\Phi_0$, respectively. We have the
following result.

\begin{proposition}\label{Pro-3.2}
Let $a_0, b_0 > 0$, $\beta_0 > 0$, $0 < \kappa_0 < 1$, and let $(u_{\kappa_0},v_{\kappa_0})$ be any positive ground state solution of \eqref{c-1}.
\begin{itemize}
\item [(1)] If $N=2$ or $3$, then $(u_{\kappa_0}, v_{\kappa_0})\rightarrow(u_{0}, v_{0})$ in $H^{1}(\mathbb{R}^{N})$ as $\kappa_0\rightarrow0$. Moreover, the following conclusions hold:
(i) If $0<\beta_{0}<\min\{a_0,b_0\}$, or $\beta_{0}>\max\{a_0,b_0\}$, or $0<a_0=b_0<1$ and $\beta_0=a_0$, then $(u_0,v_0)$ is a positive radial ground state solution of \eqref{c-1} with $\kappa_0=0$.
(ii) If $\min\{a_0,b_0\}\leq\beta_{0}\leq\max\{a_0,b_0\}$ and $a_0\neq b_0$, then one of $u_0$ and $v_0$ is zero, and $(u_0,v_0)$ is a semipositive radial ground state solution of \eqref{c-1} with $\kappa_0=0$.
(iii) If $a_0=b_0\geq1$ and $\beta_0=a_0$, then $(u_0,v_0)$ is a nonnegative radial ground state solutions of \eqref{c-1} with $\kappa_0=0$.
\item [(2)] If $N=1$, then $(u_{\kappa_0}, v_{\kappa_0})\rightharpoonup(u_{0}, v_{0})$ in $H^{1}(\mathbb{R}^{N})$ as $\kappa_0\rightarrow0$, and $(u_0, v_0)$ has the same properties as above.
\end{itemize}
\end{proposition}

\begin{proof}
For $a_0, b_0>0$, we take $\kappa_0^{n}>0$ such that
$\kappa_0^{n}\rightarrow0$ as $n\rightarrow\infty$. Let $(u_{n},
v_{n}):=(u_{\kappa_0^{n}}, v_{\kappa_0^{n}})$ denote the positive
ground state solution of \eqref{c-1} with $\kappa=\kappa_0^{n}$. We
first claim that $c_0^{\kappa_0}$ is a decreasing function on
$\kappa_0$. In fact, for $0<\kappa_0^{1}\leq\kappa_0^{2}$, we let
$(u_{1},v_{1})$ and $(u_{2},v_{2})$ denote the positive ground state
solution corresponding to $\kappa_0=\kappa_0^{1}$ and
$\kappa_0=\kappa_0^{2}$ respectively. Then there exists unique
$t_{1}(\kappa_0^{1})>0$ such that
$t_{1}(\kappa_0^{1})(u_{2},v_{2})\in\mathscr{N}_0^{\kappa_0^{1}}$.
As in \eqref{c-15} we know that $t_{1}(\kappa_0)$ is a decreasing
function on $\kappa_0$. Hence, we infer from
$(u_{2},v_{2})\in\mathscr{N}_0^{\kappa_0^{2}}$ that
$\hat{t}_{1}:=t_{1}(\kappa_0^{1})\leq t_{1}(\kappa_0^{2})=1$, where
$t_{1}(\kappa_0^{2})$ satisfies
$t_{1}(\kappa_0^{2})(u_{2},v_{2})\in\mathscr{N}_0^{\kappa_0^{2}}$.
So, it follows that
\begin{equation}\label{c-17}
\Phi_0^{\kappa_0^{1}}(u_{1},v_{1})\leq\Phi_0^{\kappa_0^{1}}(\hat{t}_{1}u_{2},\hat{t}_{1}v_{2})=\frac{\hat{t}_{1}^{2}}{6}\|(u_{2},v_2)\|_{\kappa_0}^{2}
\leq\frac{1}{6}\|(u_{2},v_2)\|_{\kappa_0}^{2}=\Phi_0^{\kappa_0^{2}}(u_{2},v_{2}).
\end{equation}
Thus, the claim holds. Hence, we know that
\begin{equation}\label{c18}
c_0^{\kappa_0^{1}}\leq
c_0^{\kappa_0^{n}}=\left(\frac{1}{2}-\frac{1}{p}\right)\|(u_{n},v_{n})\|_{\kappa_0^{n}}^2
\leq c_0^{\kappa_0^{n+1}}\leq\cdot\cdot\cdot\leq c_0^{0}.
\end{equation}
This implies that $\{(u_n, v_n)\}$ is bounded in $E$. In addition,
since $\beta_0>0$ and $\kappa_0^{n}>0$ close to zero, by the moving
plane method (see \cite[Theorem 2]{Busca-Sirakov-2000-JDE}), $u_n$
and $v_n$ must be radially symmetric and strictly decreasing
functions.

We first consider the case $2\leq N\leq3$. Without loss of
generality we assume that
$(u_{n},v_{n})\rightharpoonup(u_{0},v_{0})$ in
$E_{r}=H_{r}^{1}(\mathbb{R}^{N})\times H_{r}^{1}(\mathbb{R}^{N})$,
and $(u_{n},v_{n})\rightarrow(u_{0},v_{0})$ in
$L_{r}^{p}(\mathbb{R}^{N})\times L_{r}^{p}(\mathbb{R}^{N})$$(\forall
p\in(2,2^{*}))$. Moreover, $u_{0}, v_{0}\geq0$ in $\mathbb{R}^{N}$.
For each $(\varphi,\phi)\in C_{0}^{\infty}(\mathbb{R}^{N})\times
C_{0}^{\infty}(\mathbb{R}^{N})$, one has that
\begin{equation}\label{c19}
\begin{split}
\left(\Phi_{0}^{\kappa_0^{n}}\right)'(u_{n},v_{n})(\varphi,\phi)&=\int_{\mathbb{R}^{N}}
[\nabla u_{n}\nabla\varphi+u_{n}\varphi+\nabla
v_{n}\nabla\phi+v_{n}\phi-3(a_0u_{n}^{2}\varphi+b_0v_{n}^{2}\phi)]\\
&\quad-\beta_0\int_{\mathbb{R}^{N}}(u_{n}v_n^{2}\varphi+v_{n}u_n^{2}\phi)+\kappa_0^{n}\int_{\mathbb{R}^{N}}
(v_{n}\varphi+u_{n}\phi)\\
&=\int_{\mathbb{R}^{N}} [\nabla u\nabla\varphi+u\varphi+\nabla
v\nabla\phi+v\phi-3(a_0u^{2}\varphi+b_0v^{2}\phi)]\\
&\quad-\beta_0\int_{\mathbb{R}^{N}}(uv^{2}\varphi+vu^{2}\phi).
\end{split}
\end{equation}
Thus, $(u_{0},v_{0})$ satisfies \eqref{c-1} with $\kappa_0=0$.
Moreover, as in \eqref{c18} one infers that
\begin{equation}\label{c20}
c_0^{\kappa_0^{n}}=\Phi_0^{\kappa_0^{n}}(u_n,v_n)\rightarrow
c_0^{0}\quad\text{and}\quad
\left(\Phi_0^{\kappa_0^{n}}\right)'(u_n,v_n)=0.
\end{equation}
As in Lemma \ref{lem-3.1}, we infer that
$\int_{\mathbb{R}^{N}}(u_{n}^{4}+v_n^4)\geq
\int_{B_{R}(y_n)}(u_{n}^{4}+v_n^4)\geq\delta>0$. Thus, from
Brezis-Lieb lemma(see \cite{Willem1996book}) we infer that
$\int_{\mathbb{R}^{N}}(u_{n}^{4}+v_n^4)\rightarrow\int_{\mathbb{R}^{N}}(u_{0}^{4}+v_0^4)$,
and $\int_{\mathbb{R}^{N}}(u_{0}^{4}+v_0^4)\geq\delta>0$. Hence, we
know that at least one of $u_{0}$ and $v_0$ is not equal to zero.
Moreover, we infer from \eqref{c20} that $(u_0,v_0)$ is a
nonnegative ground state solution of \eqref{c-1} with $\kappa_0=0$.
To make it clear we divide into the following cases:
\begin{itemize}
  \item [(1)] If $\beta_{0}>\max\{a_0,b_0\}$, as in the proof of
  \cite[Theorem 1]{Sirakov-2007-CMP}, we know that \eqref{c-1} has
  positive radial ground state solution with $\kappa_0=0$. Moreover,
  $C_0^0<\frac{S_{1}^2}{4}\max\{\frac{1}{a_0^2},\frac{1}{b_0^2}\}$.
  Thus, $u_0\not\equiv0$ and $v_0\not\equiv0$ is a nonnegative
 radial ground state solution of \eqref{c-3} with $\kappa_0=0$. By
 maximum principal, one deduces that $(u_0,v_0)$ is a positive
 radial ground state solution of \eqref{c-3} with $\kappa_0=0$.

\item [(2)]If $0<\beta_{0}<\min\{a_0,b_0\}$, by using similar
 arguments as in \cite[Proposition 3.3]{Sirakov-2007-CMP}, we know
 that $\int_{\mathbb{R}^{N}}u_{n}^{4},
 \int_{\mathbb{R}^{N}}v_{n}^{4}\geq\delta_{0}>0$. Hence, $u_0\not\equiv0$ and $v_0\not\equiv0$ is a
 positive
 radial ground state solution of \eqref{c-1} with $\kappa_0=0$.

  \item [(3)] If $\min\{a_0,b_0\}\leq\beta_{0}\leq\max\{a_0,b_0\}$ and $a_0\neq
  b_0$, as in \cite[Theorem 1]{Sirakov-2007-CMP}, the system \eqref{c-1} does not have a
nontrivial solution(both component are not equal to zero) with
nonnegative components. So, one of $u_0$ and $v_0$ must be zero.

\item [(4)] If $0<a_0=b_0<1$ and $\beta_0=a_0$, we claim that both $u_0$
and $v_0$ are non zero. In fact, if $u_0=0$, we infer from
\cite[Proposition 3.2]{Sirakov-2007-CMP} that
\begin{equation}\label{c21}
\frac{S_1^2}{4a_0^2}=c_0^{0}\leq\frac{S_1^2}{4a_0}.
\end{equation}
This is a contradiction. So, $u_0\not\equiv0$ and $v_0\not\equiv0$ is
a positive radial ground state solution of \eqref{c-1} with
$\kappa_0=0$ in this case.
\end{itemize}

Next we consider the case $N=1$. As in \eqref{c18}, we obtain
$(u_{n},v_{n})$ is bounded in $E_{r}$. Without loss of generality we
assume that $(u_{n},v_{n})\rightharpoonup(u_{0},v_{0})$ in
$H_{r}^{1}(\mathbb{R})\times H_{r}^{1}(\mathbb{R})$, and
$(u_{n},v_{n})\rightarrow(u_{0},v_{0})$ in
$L_{r,loc}^{p}(\mathbb{R})\times L_{r,loc}^{p}(\mathbb{R})$$(\forall
p\in(2,\infty))$. Moreover, similar to \eqref{c19}, one deduces that
$(u_{0},v_{0})$ is an nonnegative solution of \eqref{c-1} with
$\kappa_0=0$. Since $C_0^{\kappa_0}$ is a decreasing function on
$\kappa_0$, it follows from \eqref{b-5} that
\begin{equation}\label{c22}
\frac{1}{4}\int_{\mathbb{R}}(a_{0}u_{0}^{4} +b_{0}v_{0}^{4}+2\beta_0
u_{0}^{2}v_{0}^{2})\geq c_0^{0}\geq c_0^{\kappa_{0}^n}\geq
c^{\kappa_{0}^1}\geq\delta>0.
\end{equation}
So, at least one of $u_{0}$ and $v_0$ is not equal to zero. The rest
of the proof is almost the same as the case $2\leq N\leq3$. We omit
the details here.
\end{proof}

\subsection{Uniqueness and nondegeneracy of the positive solution}

In this subsection we consider the uniqueness and nondegeneracy of the
positive solution of \eqref{c-1} when $p=4$, and $a_0$ and $b_0$ are
positive constants. Then the system is
\begin{equation}\label{auto-2}
\begin{cases}
-\Delta u+u=a_{0}u^{3}+\beta_{0}v^{2}u+\kappa_{0}v,\quad &\text{in}\ \mathbb{R}^{N}\\
-\Delta v+v=b_{0}v^{3}+\beta_{0}u^{2}v+\kappa_{0}u,\quad &\text{in}\
\mathbb{R}^{N}.
\end{cases}
\end{equation}
We look for a synchronized solution of the form
$z=(a_{1}w(a_3x),a_2w(a_3x))$, where $w$ is the unique positive
solution of
\begin{equation}\label{c-18}
-\Delta u+u=u^{3},\quad u\in H^1(\mathbb{R}^{N}).
\end{equation}
Substituting $z$ into \eqref{auto-2} gives
\begin{equation}\label{c-19}
\begin{cases}
a_1a_3^2=a_1-\kappa_{0}a_2,\\
a_3^2=a_{0}a_1^2+\beta_{0}a_2^2,\\
a_2a_3^2=a_2-\kappa_{0}a_1,\\
a_3^2=b_{0}a_2^2+\beta_0a_1^2.
\end{cases}
\end{equation}
Solving \eqref{c-19} gives
\begin{equation}\label{c-20}
\begin{cases}
a_{0}=b_{0}:=\mu>0,\\
a_1=a_2=\pm\sqrt{\frac{1-\kappa_0}{\mu+\beta_0}}, \beta_0>-\mu, 0<\kappa_0<1,\\
a_3=\sqrt{1-\kappa_0},
\end{cases}
\end{equation}
or
\begin{equation}\label{c-21}
\begin{cases}
a_{0}=b_{0}:=\mu>0,\\
a_1=-a_2=\pm\sqrt{\frac{1-\kappa_0}{\mu+\beta_0}}, \beta_0>-\mu, 0<\kappa_0<1,\\
a_3=\sqrt{1-\kappa_0}.
\end{cases}
\end{equation}
Without loss of generality we may assume that $a_{0}=b_{0}=\mu=1$.
Then the system \eqref{auto-2} has the four synchronized solutions
$z_{1,2}=\left(\pm\sqrt{\frac{1-\kappa_0}{1+\beta_0}}w(\sqrt{1-\kappa_0}x),
\pm\sqrt{\frac{1-\kappa_0}{1+\beta_0}}w(\sqrt{1-\kappa_0}x)\right)$
and
$z_{3,4}=\left(\pm\sqrt{\frac{1-\kappa_0}{1+\beta_0}}w(\sqrt{1-\kappa_0}x),
\mp\sqrt{\frac{1-\kappa_0}{1+\beta_0}}w(\sqrt{1-\kappa_0}x)\right)$.
We first consider the uniqueness of the positive solution $z_1$.
\begin{lemma}\label{lem-3.2}
\begin{enumerate}
  \item [(1)]If $N=1$ and $\mu=1$, then $z_1$ is the unique positive solution of \eqref{auto-2} in the following cases: (i) $1\leq\beta_0$ and $0<\kappa_0<1$, (ii) $-1<\beta_0<1$ and $\kappa_0>0$ is sufficiently small, (iii) $0<\kappa_0<1$ and $\beta_0-1<0$ is sufficiently small.
 \item [(2)] If $N=2$ or $3$, then $z_1$ is the unique positive solution of \eqref{auto-2} in the following cases: (i) $\beta_0\geq1$ and $0<\kappa_0<1$, (ii) $\beta_0, \kappa_0>0$ are sufficiently small, (iii) $\beta_0>0$ is sufficiently small and $\kappa_0$ is close to $1$.
\item [(3)] If $\beta_0<0$ is sufficiently small, and $\kappa_0=0$ or $\kappa_0$ is close to $1$, then $z_1$ is the unique radial positive solution of \eqref{auto-2}.
\end{enumerate}
\end{lemma}

\begin{proof}
\begin{enumerate}
  \item [(1)] We modify the argument of \cite[Theorem 1.1]{WeiYao2012-CPAA}. If $N=1$ and $\mu=1$,
system \eqref{auto} reduces to
\begin{equation}\label{c-22}
\begin{cases}
&-u''+u=u^{3}+\beta_{0}v^{2}u+\kappa_{0}v,\quad \text{in}\ [0,\infty),\\
&-v''+v=v^{3}+\beta_{0}u^{2}v+\kappa_{0}u,\quad \text{in}\
[0,\infty),\\
&u(r)>0, v(r)>0\ \text{in}\ [0,\infty),\\
&u'(0)=v'(0)=0,\ u(r),\ v(r)\rightarrow0\ \text{as}\
r\rightarrow\infty.
\end{cases}
\end{equation}
Let $(u,v)$ be a positive solution of \eqref{c-22}. Thus, we only
need to prove that $v(r)=u(r)$ for all $r\geq0$ by the uniqueness
result of the single scalar equation. Multiplying the first and
second equations of \eqref{c-22} by $v$ and $u$ respectively, then
we have that
\begin{equation}\label{c-23}
(u'v)'-u'v'-uv+u^{3}v+\beta_{0}v^{3}u+\kappa_{0}v^2=0,
\end{equation}
and
\begin{equation}\label{c-24}
(uv')'-u'v'-uv+v^{3}u+\beta_{0}u^{3}v+\kappa_{0}u^2=0.
\end{equation}
Subtracting \eqref{c-23} by \eqref{c-24} gives
\begin{equation}\label{c-25}
(u'v-uv')'+(1-\beta_0)uv(u^2-v^2)+\kappa_{0}(v^2-u^2)=0.
\end{equation}
Integrating \eqref{c-25} over $(0,\infty)$ and using
$u'(0)=v'(0)=u(\infty)=v(\infty)$, we have
\begin{equation}\label{c-26}
\int_{0}^{\infty}[(1-\beta_0)uv-\kappa_{0}](u+v)(u-v)=0.
\end{equation}
We first claim that for $\beta_0>-1$, $u\geq v$ or $u\leq v$.
Suppose not, then $g=u-v$ changes sign. It is easy to see that
\begin{equation}\label{c-27}
g''-(1+\kappa_{0})g-(u^2+uv+v^2-\beta uv)g=0\quad \text{in}\
[0,\infty).
\end{equation}
By using unique continuation property for elliptic equation, we know
that $g$ is not equal to zero in any nonempty interval.
Furthermore, by Maximum principle we infer that $g(r)=u(r)-v(r)$
changes sign only finite time. Without loss of generality we may
assume that $g(r)>0$ for large $r$. Thus there exists $R_1>0$ such
that for $r>R_1$ and $\beta_0\neq1$
\begin{equation}\label{c-28}
u(R_1)v(R_1)<\frac{\kappa_{0}}{|1-\beta_0|},\quad
u(R_1)-v(R_1)=0\quad\text{and}\quad u(r)-v(r)>0.
\end{equation}
This implies that
\begin{equation}\label{c-29}
u'(R_1)-v'(R_1)\geq0.
\end{equation}
Integrating \eqref{c-25} over ($R_1,\infty$) we obtain
\begin{equation}\label{c-30}
-(u'v-uv')(R_1)+\int_{R_1}^{\infty}[(1-\beta_0)uv-\kappa_{0}](u+v)(u-v)=0.
\end{equation}
We infer from \eqref{c-28}-\eqref{c-30} that
\begin{equation}\label{c-31}
\begin{split}
&-(u'v-uv')(R_1)=-u(R_1)[u'(R_1)-v'(R_1)]\leq0 \quad \text{and}\\
&(1-\beta_0)u(r)v(r)-\kappa_{0}\leq|1-\beta_0|u(R_1)v(R_1)-\kappa_{0}<0,\
\forall r\geq R_1.
\end{split}
\end{equation}
This contradicts with \eqref{c-30}. If $\beta_0=1$, we can also find
the contradiction by using the argument of
\eqref{c-30}-\eqref{c-31}. So, we prove the claim that for
$\beta_0>-1$, $u\geq v$ or $u\leq v$. Finally, we need prove that
$u\equiv v$. We divide into the following three cases:

\begin{enumerate}
  \item [($a$)]If $\beta_0\geq1$, we know that
\begin{equation}\label{c-32}
(1-\beta_0)u(r)v(r)-\kappa_{0}<0,\ \forall r>0.
\end{equation}
Moreover, we infer from $u\geq v$ or $u\leq v$ that the left hand of
the integral is strict less than zero. This is contradiction. Thus,
$u\equiv v$ in this case.
  \item [($b$)]For each $-1<\beta_0<1$, without loss of generality we assume
that $u\geq v$. It is clear that there exists $R_2>0$ large enough
such that
\begin{equation}\label{c-33}
\int_{0}^{R_2}(u(r)+v(r))(u(r)-v(r))dr>\int_{R_2}^{\infty}(u(r)+v(r))(u(r)-v(r))dr.
\end{equation}
Furthermore, we can choose $\kappa_0$ small enough such that
$[(1-\beta_0)u(R_2)v(R_2)-\kappa_0]>0$. So, we obtain that
\begin{equation}\label{c-34}
\begin{split}
& \int_{0}^{\infty}[(1-\beta_0)uv-\kappa_{0}](u+v)(u-v)\\
\geq& [(1-\beta_0)u(R_2)v(R_2)-\kappa_{0}]\left(\int_{0}^{R_2}(u^2-v^2)dr-\int_{R_2}^{\infty}(u^2-v^2)dr\right)\\
>& 0.
\end{split}
\end{equation}
This contradicts with \eqref{c-26}.
  \item [($c$)] For each $0<\kappa_0<1$, we should prove that the conclusion
holds if $\beta-1$ close to $0^+$. Since $u(0)=\max u(x)$ and
$v(0)=\max v(x)$, if $u(0)v(0)<\frac{\kappa_0}{1-\beta_0}$, we have
that
\begin{equation}\label{c-35}
\int_{0}^{\infty}[(1-\beta_0)uv-\kappa_{0}](u+v)(u-v)<0.
\end{equation}
This contradicts with \eqref{c-26}.
\end{enumerate}

\item [(2)] We first use the idea of \cite{WeiYao2012-CPAA} to consider the case ($i$).
Let $\Gamma_{+}=\{x\in\mathbb{R}^{N}: u(x)-v(x)>0\}$. Then
$\Gamma_{+}$ is a piecewise $C^1$ smooth domain. Multiplying the
first equation in \eqref{auto} by $v$ and the second equation in
\eqref{auto} by $u$ and then integrating by parts on $\Gamma_{+}$
and subtracting together, we obtain the following integral identity
\begin{equation}\label{a-21}
\int_{\partial\Gamma_{+}}(v\frac{\partial u}{\partial
n}-u\frac{\partial v}{\partial
n})+\int_{\Gamma_{+}}[(1-\beta_0)uv-\kappa_{0}](u+v)(u-v)=0,
\end{equation}
where $n$ denotes the unit outward normal to $\Gamma_{+}$. Since
$u(x)-v(x)>0$ in $\Gamma_{+}$, $u(x)-v(x)=0$ on $\partial\Gamma_{+}$
and
$\lim_{|x|\rightarrow\infty}u(x)=\lim_{|x|\rightarrow\infty}v(x)=0$,
it follows that
\begin{equation}\label{a-22}
\int_{\partial\Gamma_{+}}(v\frac{\partial u}{\partial
n}-u\frac{\partial v}{\partial
n})=\int_{\partial\Gamma_{+}}u\frac{\partial (u-v)}{\partial
n}\leq0.
\end{equation}
On the other hand, one sees that for $\beta\geq1$ and $0<\kappa_0<1$
\begin{equation}\label{a-23}
\int_{\Gamma_{+}}[(1-\beta_0)uv-\kappa_{0}](u+v)(u-v)<0.
\end{equation}
Hence, $\Gamma_{+}=\emptyset$. Similarly, we may prove that the set
$\Gamma_{-}=\{x\in\mathbb{R}^{N}: u(x)-v(x)>0\}$ is also an empty
set. Therefore, $u=v$ and we complete the proof of the case $(i)$.

Second, we consider the case $(ii)$. According to \cite[Theorem
4.1]{WeiYao2012-CPAA}, if $\beta_0>0$ small and $\kappa_0=0$, we
know that
$z_1=(\sqrt{\frac{1-\kappa_0}{1+\beta_0}}w(\sqrt{1-\kappa_0}x),
\sqrt{\frac{1-\kappa_0}{1+\beta_0}}w(\sqrt{1-\kappa_0}x))$ is a
unique positive solution of \eqref{auto}. Moreover, $z_1$ is
nondegenerate in $E_r=H_{r}^{1}(\mathbb{R}^{N})\times
H_{r}^{1}(\mathbb{R}^{N})$ by \cite[Lemma 2.2]{WeiYao2012-CPAA}. For
each $z=(u,v)\in E_r$ we define
\begin{equation}\label{a-19}
\Phi_{\kappa_0}(u,v)=\frac{1}{2}(\|u\|^{2}+\|v\|^{2})-\frac{1}{4}\int_{\mathbb{R}^{N}}(u^4+v^4+2\beta_0
u^2v^2)-\kappa_0\int_{\mathbb{R}^{N}}uv.
\end{equation}
Let $\Psi(\kappa_0,u,v)=\Phi'_{\kappa_0}(u,v)$. Obviously, we have
that $\Psi(0,z_1)=0$. Moreover,
$\Psi_{z}(0,z_1)=\Phi''_{\kappa_0}(z_1)$ is invertible. By the
implicit function theorem, there exist $\widetilde{\beta}_0> 0, R_0
>0$ and $\psi:(-\widetilde{\beta}_0,\widetilde{\beta}_0)\rightarrow B_{R_0}(z_1)$
such that for any
$\beta_0\in(-\widetilde{\beta}_0,\widetilde{\beta}_0)$,
$\Psi(\kappa_0,z)=0$ has a unique solution $z=\psi(\beta_0)$ in
$B_{R_0}(z_1)$. Furthermore, by using the same blow up arguments as
\cite[Lemma 2.4]{Dancer-Wei-2009-TRMS}, we know that for each fixed
$0<\kappa_0<1$, there exists $C_{\kappa_0}>0$ such that
\begin{equation}\label{a-20}
|u|_{L^{\infty}(\mathbb{R}^{N})}+|v|_{L^{\infty}(\mathbb{R}^{N})}\leq
C_{\kappa_0},
\end{equation}
where $(u,v)$ is a nonnegative solution of \eqref{auto-2}. Thus for
$\kappa_0$ sufficiently small, the set of solutions to system
\eqref{auto-2} is contained in $B_{R_0}(z_1)$.

Finally, we prove the case ($iii$). For $\bar{\beta}_0>0$, we define
\begin{equation*}
\mathcal {S}_{\bar{\beta}_0}=\{z=(u,v)\in E_{r}: z\ \text{is\ a\
positive\ solution\ of\ \eqref{auto-2}\ with}\
\beta_0\in[0,\bar{\beta}_0]\},
\end{equation*}
where $E_{r}=H_{r}^1(\mathbb{R}^{N})\times H_{r}^1(\mathbb{R}^{N})$.
By using a minor modification of the arguments of \cite[Corollary
2.4]{Ikoma-2009-Nodea}, we know that $\mathcal {S}_{\bar{\beta}_0}$
is compact in $E_r$. Moreover, according to \cite[Lemma
3.13]{ACR-1}, we know that the unique positive solution
$\tilde{z}_0=(\sqrt{1-\kappa_0}w(\sqrt{1-\kappa_0}x),
\sqrt{1-\kappa_0}w(\sqrt{1-\kappa_0}x))$ with $\beta_0=0$ of
\eqref{auto-2} is nondegenerate. So, by using the same arguments as
in the proof of the case ($i$), we can prove that for $\beta_0>0$
small, \eqref{auto-2} has a unique positive solution.

\item [(3)] Let $\mathcal {S}_{-\bar{\beta}_0}=\left\{z=(u,v)\in E_{r}: z\ \text{is\ a\
positive\ solution\ of\ \eqref{auto-2}\ with}\
\beta_0\in[\bar{\beta}_0,0]\right\}$. We first claim that for any
$\bar{\beta}_0>0$, there exists $C_{\bar{\beta}_0}>0$ such that
\begin{equation*}
|u|_{\infty}+|v|_{\infty}\leq C_{\bar{\beta}_0}.
\end{equation*}
Similarly, we also use the blow up arguments as \cite[Lemma
2.4]{Dancer-Wei-2009-TRMS}. Assume that there exist a sequence of
positive solutions $\{z_{n}=(u_{n},v_{n})\}$ of \eqref{auto-2} with
$\beta_{n}\in[-\bar{\beta}_0,0]$ such that
$\beta_{n}\rightarrow\tilde{\beta}$ and
$|v_{n}|_{\infty}\leq|u_{n}|_{\infty}\rightarrow\infty$ as
$n\to\infty$. We set
\begin{equation*}
\eta_{n}=\frac{1}{|u_{n}|_{\infty}},\quad
(w_{n}(x),h_{n}(x))=(\eta_{n}u_{n}(\sqrt{\eta_{n}}x),\eta_{n}v_{n}(\sqrt{\eta_{n}}x)).
\end{equation*}
Since $u_{n}$ and $v_{n}$ are radially symmetric and decreasing in
the radial direction. Hence $|h_{n}|_{\infty}\le
|w_{n}|_{\infty}=w_n(0)=1$. It is easy to verify that
$(w_{n},h_{n})$ satisfies
\begin{equation*}
\begin{cases}
-\Delta
w_{n}+\eta_{n}w_{n}=w_{n}^3+\beta_{n}h_{n}^2w_{n}+\eta_{n}\kappa_0h_{n},\\
-\Delta
h_{n}+\eta_{n}h_{n}=h_{n}^3+\beta_{n}w_{n}^2h_{n}+\eta_{n}\kappa_0w_{n}.
\end{cases}
\end{equation*}
By the standard elliptic argument, we may assume that, subject to a
subsequence, $(w_{n},h_{n})\rightarrow(w_{0},h_{0})$ in
$C_{loc}^{2}(\mathbb{R}^{3})$ as $n\to\infty$, where $(w_{0},h_{0})$
is a nonnegative solution of
\begin{equation}\label{c-j41}
\begin{cases}
-\Delta
w_{0}=w_{0}^3+\tilde{\beta}h_0^{2}w_{0},\\
-\Delta h_{0}=h_{0}^3+\tilde{\beta}w_0^{2}h_{0}.
\end{cases}
\end{equation}
Since $w_0(0)=1$ and the maximum principle shows that $w_0(x)>0$ in
$\mathbb{R}^N$. However, as in \cite[Theorem
2.1]{Wei-Juncheng-2010-Anals}, we know that for $\tilde{\beta}>-1$,
any nonnegative solution of \eqref{c-j41} is zero. This is a
contradiction. \qedhere
\end{enumerate}
\end{proof}

By \cite[Theorems 4.1 and 4.2]{WeiYao2012-CPAA}, we have the following result for $\kappa_0=0$.

\begin{lemma}\label{lem-3.3}
\begin{itemize}
  \item [(1)]If $N=1$ and $0\leq\beta_0\not\in[\min\{a_0,b_0\}, \max\{a_0,b_0\}]$, then $z_0=(\sqrt{\frac{\beta_0-b_0}{\beta_0^2-a_0b_0}}w(x)$, $\sqrt{\frac{\beta_0-a_0}{\beta_0^2-a_0b_0}}w(x))$ is the unique solution of \eqref{auto-2}.
  \item [(2)]If $N=2$ or $3$, and $\beta_0>\max\{a_0,b_0\}$ or $\beta_0>0$ is sufficiently small, then $z_0$ is the unique positive solution of \eqref{auto-2}.
\end{itemize}
\end{lemma}

Next we study the nondegeneracy of solutions of the system \eqref{auto-2}. Recall that $(U_1,U_2)$ is a nondegenerate solution if the solution set of the linearized system
\begin{equation}\label{a-24}
\begin{cases}
&\Delta \phi_{1}-\phi_{1}+\kappa_0\phi_2+3U_{1}^{2}\phi_1+\beta_{0}U_{2}^{2}\phi_1+2\beta_{0}U_{1}U_{2}\phi_2=0,\\
&\Delta
\phi_{2}-\phi_{2}+\kappa_0\phi_1+3U_{2}^{2}\phi_2+\beta_{0}U_{1}^{2}\phi_2+2\beta_{0}U_{1}U_{2}\phi_1=0,\\
&\phi_{1}=\phi_{1}(r),\ \phi_{2}=\phi_{2}(r)
\end{cases}
\end{equation}
is $N$-dimensional, i.e.,
\begin{equation}\label{a-25}
\phi=\left(\begin{array}{cc}
\phi_{1}\\
\phi_{2} \\
    \end{array}
  \right)
  =\sum_{j=1}^{N}k_j\left(\begin{array}{cc}
\frac{\partial U_1}{\partial x_j}\\
\frac{\partial U_2}{\partial x_j} \\
    \end{array}
  \right).
\end{equation}
Set $z_1=
(c_{0}w_{0},c_{0}w_{0}):=(\sqrt{\frac{1-\kappa_0}{1+\beta_0}}w(\sqrt{1-\kappa_0}x),$
$\sqrt{\frac{1-\kappa_0}{1+\beta_0}}w(\sqrt{1-\kappa_0}x))$. In the
following we study the nondegeneracy of the solution $z_1$. First we
have the following result for $\kappa_0=0$ as in \cite[lemma 2.2]{Dancer-Wei-2009-TRMS}.

\begin{lemma}\label{lem-3.4}
If $0\leq\beta_0\not\in[\min\{a_0,b_0\},\max\{a_0,b_0\}]$, then
$\hat{z}_1=(\sqrt{\frac{\beta_0-b_0}{\beta_0^2-a_0b_0}}w(x),$
$\sqrt{\frac{\beta_0-a_0}{\beta_0^2-a_0b_0}}w(x))$ is nondegenerate
in the space of radial functions.
\end{lemma}

For $\kappa_0\neq0$, we have the following result.

\begin{lemma}\label{lem-3.5}
Let $0<\kappa_0<1$. If $\beta_{0}\geq3$, or $-1<\beta_{0}<3$ and $w(0)\leq
\sqrt{\frac{2\kappa_0(1+\beta_0)}{(3-\beta_0)(1-\kappa_0)}}$, then
$z_1=
(c_{0}w_{0},c_{0}w_{0}):=(\sqrt{\frac{1-\kappa_0}{1+\beta_0}}w(\sqrt{1-\kappa_0}x),$
$\sqrt{\frac{1-\kappa_0}{1+\beta_0}}w(\sqrt{1-\kappa_0}x))$ is
nondegenerate in the space of radial functions, where $w(0)=\max w$
and $w$ is the unique positive solution of the scalar equation.
\end{lemma}

\begin{proof}
If $\kappa_0\neq0$, we shall prove the nondegenerate of $z_1=
(c_{0}w_{0},c_{0}w_{0}):=(\sqrt{\frac{1-\kappa_0}{1+\beta_0}}$
$w(\sqrt{1-\kappa_0}x),
\sqrt{\frac{1-\kappa_0}{1+\beta_0}}w(\sqrt{1-\kappa_0}x))$. The
linearized problem of \eqref{auto-2} at $z_1$ becomes
\begin{equation}\label{c-43}
\begin{cases}
&\Delta \phi_{1}-\phi_{1}+\kappa_0\phi_2+(3+\beta_{0})c_{0}^{2}w_{0}^{2}\phi_1+2\beta_{0} c_{0}^{2}w_{0}^{2}\phi_2=0,\\
&\Delta
\phi_{2}-\phi_{2}+\kappa_0\phi_1+(3+\beta_{0})c_{0}^{2}w_{0}^{2}\phi_2+2\beta_{0}
c_{0}^{2}w_{0}^{2}\phi_1=0,\\
&\phi_{1}=\phi_{1}(r),\ \phi_{2}=\phi_{2}(r).
\end{cases}
\end{equation}
By an orthonormal transformation, \eqref{c-43} can be transformed to
two single equations
\begin{equation}\label{c-44}
\begin{cases}
&\Delta\Phi_{1}-(1-\kappa_0)\Phi_{1}+3(1-\kappa_{0})w^{2}(\sqrt{1-\kappa_0}x)\Phi_1=0,\\
&\Delta\Phi_{2}-(1-\kappa_0)\Phi_{1}+\left[\frac{(3-\beta_0)(1-\kappa_0)}{1+\beta_0}w^{2}(\sqrt{1-\kappa_0}x)-2\kappa_0\right]\Phi_2=0.
\end{cases}
\end{equation}
By scaling $x\mapsto\frac{y}{\sqrt{1-\kappa_0}}$, we know that
\eqref{c-44} becomes
\begin{equation}\label{c-45}
\begin{cases}
&\Delta\Psi_{1}-\Psi_{1}+3w^{2}(y))\Psi_1=0,\\
&\Delta\Psi_{2}-\Psi_{2}+\left[\frac{(3-\beta_0)}{1+\beta_0}w^{2}(y)-\frac{2\kappa_0}{1-\kappa_0}\right]\Psi_2=0,
\end{cases}
\end{equation}
where $\Psi_{i}(y)=\Phi_i(\frac{y}{\sqrt{1-\kappa_0}})(i=1,2)$. On
the other hand, since the eigenvalues of
\begin{equation}\label{c-46}
\Delta\Psi-\Psi+\lambda w^{2}\Psi=0,\quad \Psi\in
H^{1}(\mathbb{R}^{N})
\end{equation}
are
\begin{equation}\label{c-47}
\lambda_1=1,\  \lambda_2=\cdot\cdot\cdot=\lambda_{N+1}=3,\
\lambda_{N+2}>3,
\end{equation}
where the eigenfunction corresponding to $\lambda_1$ is $cw$, and
the eigenfunctions corresponding to $\lambda_2$ are spanned by
$\frac{\partial w}{\partial x_{j}}(j=1,2,\cdot\cdot\cdot,N)$.

So, the first equation \eqref{c-45} has only zero solution, i.e.,
$\Psi_1=0$.

If $\beta_{0}\geq3$ and $0<\kappa_0<1$ we know that
$K(\beta_0,\kappa_0):=\frac{(3-\beta_0)}{1+\beta_0}w^{2}(y)-\frac{2\kappa_0}{1-\kappa_0}<0$.
It follows that $\Psi_2=0$.

If $-1<\beta_{0}<3$, $0<\kappa_0<1$ and
$w(0)\leq\sqrt{\frac{2\kappa_0(1+\beta_0)}{(3-\beta_0)(1-\kappa_0)}}$,
then $K(\beta_0,\kappa_0)\leq0$. Thus, $\Psi_2=0$.
\end{proof}

\begin{remark}\label{rem-3.7}
Similarly, under the same conditions of Lemma \ref{lem-3.5}, one can
prove that $z_2=
(-c_{0}w_{0},-c_{0}w_{0})=(-\sqrt{\frac{1-\kappa_0}{1+\beta_0}}w(\sqrt{1-\kappa_0}x),-\sqrt{\frac{1-\kappa_0}{1+\beta_0}}w(\sqrt{1-\kappa_0}x))$
is also nondegenerate in the space of radial functions.
\end{remark}

\section{Concentration compactness lemma}

\setcounter{section}{4} \setcounter{equation}{0}

The following profile decomposition is an immediate consequence of Theorem~3.1, that trivially adapts the reasoning for the scalar case of Corollary~3.2, from \cite{ccbook} to the Hilbert space $E=H^1(\R^N) \times H^1(\R^N)$,  equipped with the group $D$ of lattice translations $D=\{(u,v)\mapsto (u(\cdot-y),v(\cdot-y)),\;y\in\mathbb \Z^N\}$.

\begin{theorem} \label{profdec} Let $(u_k, v_k)$ be a bounded sequence in $E$. There exists a renamed subsequence and a sequence $(y_k^{(n)})_k\subset \mathbb \Z^N)$, $n\in\mathbb N$, such that $y_k^{(1)}=0$,
\begin{equation}\label{wlim}
U^{(n)}=\stackrel{\rightharpoonup}{\lim} u_k(\cdot +y_k^{(n)});\; V^{(n)}=\stackrel{\rightharpoonup}{\lim} v_k(\cdot +y_k^{(n)});
\end{equation}
\begin{equation}
|y_k^{(m)}-y_k^{(n)}|\to\infty\text{ for } m\neq n,
\end{equation}
\begin{equation}
\sum_n \|U^{(n)}\|^2\le \|u_k\|^2+o(1); \; \sum_n \|V^{(n)}\|^2\le \|v_k\|^2+o(1);
\end{equation}
and for any $p\in(2,2^*)$,
\begin{equation}\label{pd}
\rho_k:=u_k- \sum_n U^{(n)}(\cdot -y_k^{(n)})\to 0 \text{ in } L^p;\;
\tau_k:=v_k- \sum_n V^{(n)}(\cdot -y_k^{(n)})\to 0 \text{ in } L^p;\;
\end{equation}
and the series in the last relations are convergent unconditionally in $H^1$ and uniformly with respect to $k$.
\end{theorem}

Note that it is  a priori possible that one of the components of $(U^{(n)},V^{(n)})$ is zero.

We will now evaluate the asymptotic value of the functional \eqref{b-1} on a sequence provided by the theorem above.

\begin{proposition} \label{prop:energy}
Let $\Phi$ be the functional \eqref{b-1} and let $\Phi_0$ be the functional \eqref{c-2}. Let $(u_k,v_k)$ be the sequence provided by Theorem~\ref{profdec}. Then
\begin{equation} \label{eq:energy}
\Phi(u_k,v_k) \ge \Phi(U^{(1)},V^{(1)}) + \sum_{n=2}^\infty \Phi_0(U^{(n)},V^{(n)}).
\end{equation}
Moreover, if, in addition, $\Phi'(u_k,v_k)\to 0$ and $\Phi(u_k,v_k) \to c \in \R$, then $(U^{(1)},V^{(1)})$ is a critical point of the functional $\Phi$, $(U^{(n)},V^{(n)})$ for any $n\ge 2$ is a critical point of the functional $\Phi_0$, and
\begin{equation} \label{eq:energy2}
\Phi(U^{(1)},V^{(1)}) + \sum_{n=2}^\infty \Phi_0(U^{(n)},V^{(n)}) = c.
\end{equation}
\end{proposition}

\begin{proof}
By continuity of $\Phi-\Phi_0$ with respect to the weak convergence it suffices to prove \eqref{eq:energy} for $\Phi=\Phi_0$, which can be immediately obtained by iteration of the Brezis-Lieb lemma (see \cite{CwiTi}, Appendix B, for the scalar case). Since the map $\Phi'-\Phi'_0$ is continuous with respect to the weak convergence, the conclusion that $(U^{(n)},V^{(n)})$ is a critical point for respective functional is immediate. In order to show \eqref{eq:energy2}, let $\rho_k$ and $\tau_k$ be as in \eqref{pd} and note that $\Phi'(u,v)=(u,v)+ \varphi'(u,v)$ with continuous $\varphi'(u,v):L^p\times L^p\to E$. From here and from the criticality of points $(U^{(n)},V^{(n)})$ it follows from $(u_k,v_k)+ \varphi'(u_k,v_k)=\Phi'(u_k,v_k)\to 0$ by a standard continuity argument that $(\rho_k,\tau_k)\to 0$ in $E$. Consequently, recalling again that $\Phi-\Phi_0$ is weakly continuous, we have
\begin{eqnarray*}
c & = &\lim \Phi(u_k,v_k)\\
& = &\lim \Phi(\sum_n U^{(n)}(\cdot -y_k^{(n)}), \sum_n V^{(n)}(\cdot -y_k^{(n)}))\\
&= &\lim \Phi_0(\sum_n U^{(n)}(\cdot -y_k^{(n)}), \sum_n V^{(n)}(\cdot -y_k^{(n)}))+(\Phi-\Phi_0)(U^{(1)},V^{(1)})\\
&= &\sum \Phi_0(U^{(n)},V^{(n)})+(\Phi-\Phi_0)(U^{(1)},V^{(1)}),
\end{eqnarray*}
which proves \eqref{eq:energy2}.
\end{proof}

\section{Ground state solutions - functional-analytic setting}

In this section we study existence of ground state solutions for a functional-analytic model of our problem. We identify, for a class of functionals defined below, the ground state with the mountain pass solution. We then formulate a sufficient condition for existence of a ground state in terms of comparison with the problem at infinity (which, in these general settings, is not required itself to admit a ground state). Verification of the comparison condition and existence of the ground state for the problem at infinity is a subject of the next section, where more specific properties of the functional are invoked.
The number $p>2$ remains fixed throughout the section.

\begin{definition} Let $H$ be a Hilbert space. We say that a functional $\Phi\in C^2(H)$ is of class ${\mathcal S}_p$, if it is of the form $\Phi(u)=\frac12\|u\|^2-\frac{1}{p}\psi(u)$,
where the functional $\psi\in C^2(H)$ is bounded on bounded sets, homogeneous of degree $p$ and positive except at $u=0$, the norm refers to any of equivalent norms of $H$, and assume that $\Phi'$ is weak-to-weak continuous on $H$.  It is to be understood that the norm $\|\cdot\|$ is not fixed, but is one of equivalent norms of $H$, that may vary for different functionals in the class.
\end{definition}

Note that the functional \eqref{b-1}, and consequently \eqref{c-2}, are of the class  ${\mathcal S}_p$.

\begin{lemma}
\label{comp}
Let $\Phi\in {\mathcal S_p}$ and
let ${\mathcal N}=\{u\in H\setminus\{0\}: (\Phi'(u),u)=0\}$. Then $w\in H\setminus\{0\}$ minimizes $\Phi$ on $\mathcal N$ if and only if the path $t\mapsto tw$, $0\le t<\infty$ minimizes
\begin{equation}
\label{abstr-mpg}
c=\inf_{\eta\in P}\max_{t>0}\Phi(\eta(t)),
\end{equation}
where
\begin{equation}
P=\{\eta\in C([0,\infty);H), \eta(0)=0, \Phi(\eta(+\infty))=-\infty\}.
\end{equation}
Moreover, $w$ is a critical point of $\Phi$.
\end{lemma}
\begin{proof}
1. First note that $\Phi$ has the classical mountain pass geometry. Note also that since $\psi$ is bounded on bounded sets and homogeneous, $0\le\psi(u)\le C\|u\|^p$ which implies that $\mathcal N$ is bounded away from the origin.  Furthermore,  by Euler theorem for homogeneous functions, $(\psi{''}(u)u,u)=p(p-1)\psi(u)>0$ unless $u=0$, and therefore a minimizer $w$ of $\Phi$ on $\mathcal N$ is a nonzero critical point of $\Phi$.

2. Note that every path in $P$ intersects $\mathcal N$, which implies that $c\ge \Phi(w)$. On the other hand, $c\le \max_t \Phi(tw)=\Phi(w)$.
It is immediate then that whenever $w$ is a minimizer of  $\Phi$ on $\mathcal N$, the path $t\mapsto tw$ minimizes \eqref{abstr-mpg}.

3. Conversely, if $w\in H\setminus\{0\}$ is such that the path $t\mapsto tw$ minimizes \eqref{abstr-mpg}, then the maximum of $\Phi$ on the path is necessarily a critical point of $\Phi$, and thus belongs to $\mathcal N$, and consequently is attained at $t=1$, so $w$ is a critical point of $\Phi$. If, however, $w$ is not a minimal point of $\Phi$ on $\mathcal N$, and $w_1$ is such a minimizer, the maximum of $\Phi$ on $tw_1$ will be smaller than $c$, a contradiction.
\end{proof}

\begin{lemma}
\label{lem:nonzero}
Let $\Phi\in\mathcal{S}_p$. Let $c$ be the minimax value \eqref{abstr-mpg}.
Then, if  $\Phi(w_k)\to c$ and $\Phi'(w_k)\to 0$ in $H$, then $w_k$ is bounded. Moreover, if $w_k\rightharpoonup w\neq 0$, then $w_k\to w$ in the norm of $H$ and $w$ is a ground state of $\Phi$.
\end{lemma}
\begin{proof} The proof of the first assertion of the lemma follows the classical argument of Ambrosetti-Rabinowitz. Multiplication of $\Phi'(w_k)\to 0$ by $w_k/\|w_k\|$ (the case when $w_k=0$ on a subsequence is trivial) gives
 $$ \|w_k\|-\frac{\psi(w_k)}{\|w_k\|}\to 0.$$
This implies that $\psi(w_k)=\|w_k\|(\|w_k\|+o(1))$
and $\Phi(w_k)=(\frac{1}{2}-\frac{1}{p})\|w_k\|^2+o(\|w_k\|)\to c$, which implies that $w_k$ is bounded in norm. Then we also have $\Phi(w_k)=(\frac{1}{2}-\frac{1}{p})\|w_k\|^2+o(1)$.

We now prove the second assertion when $w_k\rightharpoonup w\neq 0$. By weak semicontinuity of the norm, $$\Phi(w_k)=(\frac{1}{2}-\frac{1}{p})\|w_k\|^2+o(1)\ge (\frac{1}{2}-\frac{1}{p})\|w\|^2+o(1)=\Phi(w)+o(1).$$
By weak-to-weak continuity of $\Phi'$, the element $w$ is a (nonzero) critical point of $\Phi$, and thus $w\in\mathcal N$. Evaluation of the functional on the path $t\mapsto tw$ gives $\lim \Phi(w_k)=c\le \Phi(w)$. Together with the previous inequality we have that $\Phi(w_k)\to \Phi(w)$. This implies that
$\|w_k\|^2\to \|w\|^2$, which in turn means that $w_k\to w$ in $H$. By Lemma~\ref{comp}, the element $w$ is a ground state of $\Phi$.
\end{proof}
Lemma~\ref{lem:nonzero} shows that the ground state exists as long as the critical sequence at the mountain pass level does not converge weakly to zero. The next lemma introduces a (still implicit) sufficient condition for the latter and thus for the existence of a ground state.

From now on we assume that $H$ is a space of functions  $\R^N\to \R^m$, $m\in\N$, such that for every sequence $y_k\in \mathbb{R}^N$, $|y_k|\to\infty $ and every $w\in H$, $w(\cdot-y_k)\rightharpoonup 0$.

\begin{definition}
Let $\Phi\in\mathcal{S}_p$. One says that a functional $\Phi_0\in\mathcal{S}_p$  is a limit of $\Phi$ at infinity if for every $y\in \mathbb{Z}^N$ and every $w\in H$,
$\Phi_0(w(\cdot+y))= \Phi_0(w)$,
and the maps $\Phi-\Phi_0$ and $\Phi'-\Phi'_0$ are continuous with respect to weak convergence.
\end{definition}

Note that $c\le c_0$, which is easy to show by evaluating $\Phi$ on the paths approximating $c_0$ for $\Phi_0$, translated by $y$ far enough from the origin of $\Z^N$ (on which the difference between $\Phi_0$ and $\Phi$ is insignificant).

In what follows we will make the following assumption on the functional $\psi_0$:
\begin{equation}
\label{coco}
w_k\in H, y_k\in\Z^N, w_k(\cdot-y_k)\rightharpoonup 0 \Longrightarrow \psi_0(w_k)\to 0.
\end{equation}
Note that our notation is consistent with the notation in the previous sections in the sense that whenever $0<\kappa_0<1$,  the functional \eqref{c-2} is the limit at infinity of the functional \eqref{b-1}.

\begin{lemma} Assume that  $\Phi\in\mathcal{S}_p$ has a limit functional $\Phi_0\in\mathcal{S}_p$ at infinity and that \eqref{coco} is satisfied.
Let $c$ and $c_0$ be the mountain pass values \eqref{abstr-mpg} for $\Phi$
and $\Phi_0$, respectively. If $c<c_0$, then $\Phi$ has a ground state. \end{lemma}
\begin{proof}
Let $w_k\rightharpoonup 0$ be a critical sequence for $\Phi$ with $\Phi(w_k)\to c$.
By definition of the functional at infinity,  we have $c=\lim \Phi(w_k)=\lim \Phi_0(w_k)=\lim \Phi_0(w_k(\cdot-y_k))$ for any sequence $y_k\in\Z^N$ such that $|y_k|\to\infty $. Note also that we may assume that $w_k(\cdot-y_k)$ has a weak limit $w_0$, which is necessarily a critical point of $\Phi_0$.
We may also assume that $w_0\neq 0$, since if it would happen for any sequence $y_k$ with $|y_k|\to\infty $, by \eqref{coco} we would have $\Phi_0(w_k)\to 0$, and thus $c=\lim\Phi(w_k)= 0$, a contradiction.
Then, by weak-to-weak continuity of $\Phi_0'$ (assured by the definition of $\mathcal{S}_p$) we have $w_0\in \mathcal{N}_0$,  and therefore $\Phi_0(w_0)\ge c_0$.

Since $\Phi_0(w_k(\cdot-y_k))=(\frac{1}{2}-\frac{1}{p})\|w_k\|_0^2+o(1)$ and the norm is weakly lower semicontinuous, we have $$c+o(1)=\Phi_0(w_k(\cdot-y_k))\ge (\frac{1}{2}-\frac{1}{p})\|w_0\|_0^2=\Phi_0(w_0)\ge c_0,$$
which contradicts the assumption  $c<c_0$.
We conclude that the critical sequence has to have a subsequence with a nonzero weak limit, which by Lemma~\ref{lem:nonzero} implies existence of a ground state.
\end{proof}

\begin{lemma}
\label{lem:less}
Assume that  $\Phi\in\mathcal{S}_p$ has a limit functional $\Phi_0\in\mathcal{S}_p$ at infinity.
Let $c$ and $c_0$ be the mountain pass values \eqref{abstr-mpg} for $\Phi$
and $\Phi_0$, respectively. Assume that $\Phi_0$ has a ground state $w_0$. If
\begin{equation}\label{eq:less}
(\frac{1}{2}-\frac{1}{p})\left(\frac{\|w_0\|^p}{\psi(w_0)}\right)^\frac{2}{p-2}<\Phi_0(w_0),
\end{equation}
then $c<c_0$.
 \end{lemma}
\begin{proof}
By defintion, $$c\le\max_{t>0} \Phi(tw_0)=\max_{t>0}\frac{t^2}{2}\|w_0\|^2-\frac{t^p}{p}\psi(w_0).$$ Elementary evaluation of the maximum gives
$$
c\le (\frac{1}{2}-\frac{1}{p})\left(\frac{\|w_0\|^p}{\psi(w_0)}\right)^\frac{2}{p-2}.
$$
Since $c_0=\Phi_0(w_0)$, the inequality above and \eqref{eq:less} imply $c<c_0$.
\end{proof}
\begin{corollary}\label{cor:less}
Under conditions of Lemma~\ref{lem:less}, if
\begin{equation}\label{eq:less2}
\frac{\|w_0\|^p}{\psi(w_0)}<\frac{\|w_0\|_0^p}{\psi_0(w_0)},
\end{equation}
then $c<c_0$.
\end{corollary}
\begin{proof}
Since $\Phi_0\in\mathcal{S}_p$, we have the representation  $\Phi_0=\frac{1}{2}\|\cdot\|_0^2-\frac{1}{p}\psi_0$ which, by repeating calculations in the proof of Lemma~\ref{lem:less} allows to represent the ground state value of $\Phi_0$ as
$$c_0=\max_t\Phi_0(tw_0)=(\frac{1}{2}-\frac{1}{p})\left(\frac{\|w_0\|_0^p}{\psi_0(w_0)}\right)^\frac{2}{p-2}.$$
Thus \eqref{eq:less} can be written in the equivalent form as  \eqref{eq:less2}.
\end{proof}

\section{Ground state solutions}

Now we give several sufficient conditions to have \eqref{eq:less2}. We will always assume that $0<\kappa_0<1$, so that the functional $\Psi_0$ has a positive ground state, which we denote as $(u,v)$, by Lemma \ref{lem-3.1}.
\begin{theorem}\label{zh0127}
 The functional \eqref{b-1} has a ground state if one of the following conditions \eqref{01}-\eqref{04}  holds.
\begin{equation}\label{01}
\left(\frac{\|(u,v)\|^2_0-\int\kappa(x)uv}{\|z\|^2_0}\right)^{p/2}<\frac{\psi(u,v)}{\psi_0(u,v)};
\end{equation}
\begin{equation} \label{02}
\kappa(x)\ge 0, a(x)\ge 0, b(x)\ge 0, \beta(x)\ge 0,
\end{equation}
provided that at least one of the inequalities is strict on a set of positive measure;

For the case $u=v$ (see sufficient conditions in Section 3),
\begin{equation} \label{03}
(a(x)+b(x)+2\beta(x))u^{p-2}+p\kappa(x)\ge 0;
\end{equation}
or in particular if
\begin{equation}
\label{04}
\kappa(x)\ge 0 \text{ and } a(x)+b(x)+2\beta(x)\ge 0,
\end{equation}
provided that at least one of the inequalities is strict on a set of positive measure.
\end{theorem}
\begin{proof}
Condition \eqref{01} is a restatement of \eqref{eq:less2}. Condition \eqref{02} obviously implies \eqref{01}. Condition \eqref{03} is \eqref{01} restated for $u=v=w$, and condition \eqref{04} trivially implies \eqref{03}.
\end{proof}

\begin{proof}[Proof of Theorem \ref{th1.1}]
From the results of Theorem \ref{zh0127}, we know that Theorem
\ref{th1.1} hold.
\end{proof}

Finally we give another proof for Lemma \ref{lem-3.1}.

\begin{proof}[Proof of Lemma \ref{lem-3.1}]
Let $w_k$ be the subsequence given by Theorem~\ref{profdec} of a
critical sequence for $\Phi_0$. Assume that the series \eqref{pd}
contains at least two nonzero terms. Then, using
Proposition~\ref{prop:energy} we have
\begin{equation}
\begin{array}{lll}
 c_0&=&\Phi_0(w_k)+o(1)=(\frac{1}{2}-\frac{1}{p})\psi_0(w_k)+o(1)\\
& \ge &(\frac{1}{2}-\frac{1}{p})(\psi_0(U^{(1)},V^{(1)})+\psi_0(U^{(2)},V^{(2)}))+o(1)\\
&=& \Phi_0(U^{(1)},V^{(1)})+\Phi_0(U^{(2)},V^{(2)})+o(1).
\end{array}
\end{equation}

At the same time, considering the functional $\Phi_0$ on the path
$t\mapsto t(U^{(1)},V^{(1)})$, we have $c_0\le \Phi_0(U^{(1)},V^{(1)})$, which is a contradiction. We conclude therefore that the critical sequence
$\tilde w_k=w_k(\cdot+y_k^{(1)})$ is convergent in $L^p$ to $(U^{(1)},V^{(1)})$, from which one can easily conclude that $(U^{(1)},V^{(1)})$ is a ground state.
\end{proof}

\section{Bound state solutions}

Throughout this section we fix $0 < \kappa_0 < 1$, and assume that $\kappa(x) \le 0$, $a(x) \le 0$, $b(x) \le 0$ and $\beta(x) \le 0$, with at least one of the inequalities strict on a set of positive measure.

\begin{lemma} \label{Lemma 7.1}
$c = c_0$ and $c$ is not attained.
\end{lemma}

\begin{proof}
The periodic system \eqref{c-1} has a positive ground state solution $(u,v) \in \mathscr{N}_0$ by Lemma \ref{lem-3.1}. Take a sequence $y_k \in \Z^N$ such that $|y_k| \to \infty$ and set
\[
(u_k,v_k) = (t_k\, u(\cdot + y_k),t_k\, v(\cdot + y_k)),
\]
where $t_k > 0$ is such that $(u_k,v_k) \in \mathscr{N}$, i.e.,
\begin{eqnarray*}
t_k^{p-2} & = & \frac{\|u(\cdot + y_k)\|^2 + \|v(\cdot + y_k)\|^2 - 2 \dint_{\R^N} (\kappa_0 + \kappa(x))\, u(x + y_k)\, v(x + y_k)\, dx}{\begin{split}
\dint_{\R^N} \big[(a_0(x) + a(x))\, u(x + y_k)^p + (b_0(x) + b(x))\, v(x + y_k)^p\\[-10pt]
+ 2\, (\beta_0 + \beta(x))\, u(x + y_k)^{p/2}\, v(x + y_k)^{p/2}\big] dx
\end{split}}\\[10pt]
& = & \frac{\|u\|^2 + \|v\|^2 - 2 \dint_{\R^N} (\kappa_0 + \kappa(x - y_k))\, uv\, dx}{\begin{split}
\dint_{\R^N} \big[(a_0(x) + a(x - y_k))\, u^p + (b_0(x) + b(x - y_k))\, v^p\\[-10pt]
+ 2\, (\beta_0 + \beta(x - y_k))\, u^{p/2}\, v^{p/2}\big] dx
\end{split}}\\[10pt]
& = & \frac{\dint_{\R^N} \big[a_0(x)\, u^p + b_0(x)\, v^p + 2 \beta_0\, u^{p/2}\, v^{p/2}\big] dx
- 2 \dint_{\R^N} \kappa(x - y_k)\, uv\, dx}{\begin{split}
\dint_{\R^N} \big[a_0(x)\, u^p + b_0(x)\, v^p + 2 \beta_0\, u^{p/2}\, v^{p/2}\big] dx
+ \dint_{\R^N} \big[a(x - y_k)\, u^p + b(x - y_k)\, v^p\\[-10pt]
+ 2 \beta(x - y_k)\, u^{p/2}\, v^{p/2}\big] dx
\end{split}}.
\end{eqnarray*}
Since $|y_k| \to \infty$, $t_k \to 1$ and hence $\Phi(u_k,v_k) \to c_0$, so $c \le c_0$. To see that the reverse inequality holds, for any $(u,v) \in \mathscr{N}$ such that $uv \ge 0$, let $t > 0$ be such that $(tu,tv) \in \mathscr{N}_0$, i.e.,
\begin{eqnarray} \label{7.1}
&~& t^{p-2}  =  \frac{\|u\|^2 + \|v\|^2 - 2 \kappa_0 \dint_{\R^N} uv\, dx}{\dint_{\R^N} \big[a_0(x)\, |u|^p + b_0(x)\, |v|^p + 2 \beta_0\, |u|^{p/2}\, |v|^{p/2}\big] dx} \notag\\[10pt]
& = & \frac{\begin{split}
\dint_{\R^N} \big[(a_0(x) + a(x))\, |u|^p + (b_0(x) + b(x))\, |v|^p + 2\, (\beta_0 + \beta(x))\, |u|^{p/2}\, |v|^{p/2}\big] dx\\[-10pt]
 + 2 \dint_{\R^N} \kappa(x)\, uv\, dx
\end{split}}{\dint_{\R^N} \big[a_0(x)\, |u|^p + b_0(x)\, |v|^p + 2 \beta_0\, |u|^{p/2}\, |v|^{p/2}\big] dx}.
\end{eqnarray}
Since $\kappa(x), a(x), b(x), \beta(x) \le 0$ and $uv \ge 0$, $t \le 1$ and hence
\begin{multline*}
c_0 \le \Phi_0(tu,tv) = t^2 \left(\frac{1}{2} - \frac{1}{p}\right) \left[\|u\|^2 + \|v\|^2 - 2 \kappa_0 \int_{\R^N} uv\, dx\right] \le \left(\frac{1}{2} - \frac{1}{p}\right) \bigg[\|u\|^2 + \|v\|^2\\[10pt]
- 2 \kappa_0 \int_{\R^N} uv\, dx\bigg]
\le \left(\frac{1}{2} - \frac{1}{p}\right) \left[\|u\|^2 + \|v\|^2 - 2 \int_{\R^N} (\kappa_0 + \kappa(x))\, uv\, dx\right] = \Phi(u,v)
\end{multline*}
by \eqref{b-5}, so $c_0 \le c$. If $\Phi(u,v) = c$, then equality holds throughout and hence $t = 1$ and $(u,v)$ is a ground state solution of system \eqref{c-1}, so an argument similar to that in the proof of Lemma \ref{lem-3.1} shows that $uv > 0$. Then \eqref{7.1} implies that $\kappa(x), a(x), b(x), \beta(x) \equiv 0$, which is contrary to assumptions.
\end{proof}

By Lemma \ref{Lemma 7.1}, system \eqref{auto} has no solutions at the level $c$, so we look for a solution at a higher energy level using the notion of barycenter as in \cite{ACR-1}. In view of Lemma \ref{lem-3.5}, we only consider the case where $p = 4$ and $a_0, b_0$ are positive constants, so our system is
\begin{equation} \label{7.2}
\begin{cases}
- \Delta u + u = (a_0 + a(x))\, u^3 + (\beta_0 + \beta(x))\, v^2 u + (\kappa_0 + \kappa(x))\, v \quad & \text{in } \R^N,\\
- \Delta v + v = (b_0 + b(x))\, v^3 + (\beta_0 + \beta(x))\, u^2 v + (\kappa_0 + \kappa(x))\, u \quad & \text{in } \R^N,
\end{cases}
\end{equation}
where $N \le 3$, $\beta_0 \in \R$, $a, b, \beta, \kappa \in L^\infty(\mathbb{R}^N)$ go to zero as $|x| \to \infty$, and
\begin{multline} \label{7.10}
a_0 + \inf_{x \in \R^N}\, a(x) > 0,~ b_0 + \inf_{x \in \R^N}\, b(x) > 0,
0 < \kappa_0 + \inf_{x \in \R^N}\, \kappa(x) \le \kappa_0 + \sup_{x \in \R^N}\, \kappa(x) < 1.
\end{multline}
The associated energy functional is
\[
\begin{split}
\Phi(u,v) & = \frac{1}{2} \left(\|u\|^2 + \|v\|^2\right) - \frac{1}{4} \int_{\R^N} \left[(a_0 + a(x))\, u^4 + (b_0 + b(x))\, v^4\right]\\
& \quad - \frac{1}{2} \int_{\R^N} (\beta_0 + \beta(x))\, u^2 v^2 - \int_{\R^N} (\kappa_0 + \kappa(x))\, uv, \quad (u,v) \in E.
\end{split}
\]

For $u \in H^1(\R^N)$, let
\[
\mu(u)(x) = \frac{1}{|B_1|} \int_{B_1(x)} |u(y)|\, dy
\]
and note that $\mu(u)$ is a bounded continuous function on $\R^N$. Then set
\[
\hat{u}(x) = \left[\mu(u)(x) - \frac{1}{2}\, \max \mu(u)\right]^+,
\]
so that $\hat{u} \in C_0(\R^N)$. The barycenter of a pair $(u,v) \in E \setminus \{(0,0)\}$ was defined in \cite{ACR-1} by
\[
\xi(u,v) = \frac{1}{|\hat{u}|_1 + |\hat{v}|_1} \int_{\R^N} x \left(\hat{u}(x) + \hat{v}(x)\right) dx
\]
(see also \cite{MR1989833}). Since $\hat{u}$ and $\hat{v}$ have compact supports, $\xi : E \setminus \{(0,0)\} \to \R^N$ is a well-defined continuous map. As noted in \cite{ACR-1}, it has the following properties:
\begin{enumerate}
\item If $u$ and $v$ are radial functions, then $\xi(u,v) = 0$.
\item For $t \ne 0$, $\xi(tu,tv) = \xi(u,v)$.
\item For all $y \in \R^N$, $\xi(u(\cdot + y),v(\cdot + y)) = \xi(u,v) - y$.
\end{enumerate}

Set
\[
\tilde{c} = \inf_{(u,v) \in \mathscr{N}\!,\, \xi(u,v) = 0}\, \Phi(u,v).
\]
Clearly, $\tilde{c} \ge c$. As in \cite{ACR-1}, we have the following lemma.

\begin{lemma} \label{Lemma 7.2}
$\tilde{c} > c$.
\end{lemma}

\begin{proof}
Suppose $\tilde{c} = c$. Then there exists a sequence $(u_k,v_k) \in \mathscr{N}$ such that $\xi(u_k,v_k) = 0$ and $\Phi(u_k,v_k) \to c$. By Ekeland's variational principle (see \cite{MR0346619}), there exists another sequence $(\tilde{u}_k,\tilde{v}_k) \in \mathscr{N}$ such that
\begin{enumerate}
\item $\Phi(\tilde{u}_k,\tilde{v}_k) \to c$,
\item $\Phi|_{\mathscr{N}}'(\tilde{u}_k,\tilde{v}_k) \to 0$,
\item $\|(\tilde{u}_k,\tilde{v}_k) - (u_k,v_k)\| \to 0$.
\end{enumerate}
As in the proof of Lemma \ref{lem-3.1}, $\Phi'(\tilde{u}_k,\tilde{v}_k) \to 0$. This together with the mean value theorem and (3) above imply that
\begin{equation} \label{7.3}
\Phi'(u_k,v_k) \to 0
\end{equation}
since $\Phi''$ maps bounded sets onto bounded sets.

Since $\kappa_0 + \sup \kappa(x) < 1$, it follows from \eqref{b-5} that $(u_k, v_k)$ is bounded. We pass to the renamed subsequence provided by Theorem \ref{profdec}, and note that
\begin{equation} \label{7.4}
(u_k,v_k) - \sum_n\, (U^{(n)}(\cdot - y_k^{(n)}),V^{(n)}(\cdot - y_k^{(n)})) \to 0 \text{ in } E
\end{equation}
by \eqref{pd}, \eqref{7.3}, and the continuity of the Sobolev imbedding. By Proposition \ref{prop:energy},
\[
\Phi(U^{(1)},V^{(1)}) + \sum_{n \ge 2}\, \Phi_0(U^{(n)},V^{(n)}) = c,
\]
$(U^{(1)},V^{(1)}) \in \mathscr{N} \cup \{(0,0)\}$, and $(U^{(n)},V^{(n)}) \in \mathscr{N}_0 \cup \{(0,0)\}$ for $n \ge 2$. In view of Lemma \ref{Lemma 7.1}, then $(U^{(1)},V^{(1)})$ trivial and $(U^{(n)},V^{(n)})$ is nontrivial for at most one $n \ge 2$. Since $\mathscr{N}$ is bounded away from the origin, then \eqref{7.4} implies that $(U^{(n)},V^{(n)})$ is nontrivial for exactly one $n \ge 2$, say, for $n_0$, which then is a ground state solution of system \eqref{c-1}. Then $(u_k,v_k) - (U^{(n_0)}(\cdot - y_k^{(n_0)}),V^{(n_0)}(\cdot - y_k^{(n_0)})) \to 0$ and hence $(u_k(\cdot + y_k^{(n_0)}),v_k(\cdot + y_k^{(n_0)})) \to (U^{(n_0)},V^{(n_0)})$ after a translation, so
\begin{equation} \label{7.5}
\xi(u_k(\cdot + y_k^{(n_0)}),v_k(\cdot + y_k^{(n_0)})) \to \xi(U^{(n_0)},V^{(n_0)})
\end{equation}
by the continuity of the barycenter. However,
\[
\xi(u_k(\cdot + y_k^{(n_0)}),v_k(\cdot + y_k^{(n_0)})) = \xi(u_k,v_k) - y_k^{(n_0)} =  - y_k^{(n_0)}
\]
and $|y_k^{(n_0)}| \to \infty$, contradicting \eqref{7.5}.
\end{proof}

Let $(u,v)$ be a radially symmetric positive ground state solution of system \eqref{c-1} and consider the continuous map
\[
\Gamma : \R^N \to \mathscr{N}, \quad \Gamma(y) = (t_y\, u(\cdot - y),t_y\, v(\cdot - y)),
\]
where $t_y > 0$ is such that $\Gamma(y) \in \mathscr{N}$. We have
\begin{equation} \label{7.11}
\xi(\Gamma(y)) = \xi(u,v) + y = y
\end{equation}
since $u$ and $v$ are radial functions. Moreover, as in the proof of Lemma \ref{Lemma 7.1},
\begin{equation} \label{7.6}
t_y^2 = \frac{\|u\|^2 + \|v\|^2 - 2 \dint_{\R^N} (\kappa_0 + \kappa(x + y))\, uv\, dx}{\dint_{\R^N} \left[(a_0 + a(x + y))\, u^4 + (b_0 + b(x + y))\, v^4 + 2\, (\beta_0 + \beta(x + y))\, u^2 v^2\right] dx}
\end{equation}
and
\begin{equation} \label{7.8}
\Phi(\Gamma(y)) \to c_0 = c \quad \text{as } |y| \to \infty.
\end{equation}

\begin{lemma} \label{Lemma 7.3}
Assume that $0<\kappa_0<1$ and one of the following conditions holds:
\begin{itemize}
  \item [(1)] $\beta_0\geq3$;
  \item [(2)] $1\leq\beta_0\leq3$ and $w(0)\leq
\sqrt{\frac{2\kappa_0(1+\beta_0)}{(3-\beta_0)(1-\kappa_0)}}$, where
$w$ is the unique positive solution of the scalar equation \eqref{c-18};
  \item [(3)] $-1<\beta_0<1$, $w(0)\leq
\sqrt{\frac{2\kappa_0(1+\beta_0)}{(3-\beta_0)(1-\kappa_0)}}$, and
\begin{itemize}
\item if $N=1$, then $\kappa_0$ or $\beta_0-1$ is sufficiently small,
\item if $N=2,3$, then $\beta_0, \kappa_0>0$ are sufficiently small, or $|\beta_0|$ is sufficiently small and $\kappa_0$ is close to 1.
\end{itemize}
\end{itemize}
Then $c_0$ is an isolated critical level of $\Phi_0$.
\end{lemma}

\begin{proof}
We use an indirect argument. Suppose that there exists a sequence $(u_n,v_n)\in
E$ such that ($i$) $\Phi'_0(u_n,v_n)=0$; $(ii)$
$\Phi_0(u_n,v_n)>c_0$; ($iii$) $\Phi_0(u_n,v_n)\rightarrow c_0$. As
in Proposition \ref{prop:energy}, there exists
$y_n\in\mathbb{R}^{N}$ such that
$(\tilde{u}_n(x),\tilde{v}_n(x))=(u_n(x-y_n),v_n(x-y_n))$ converges
to some $(u,v) \in E$ strongly. Moreover, $\Phi'_0(u,v)=0$ and
$\Phi_0(u,v)=c_0$, i.e., $z=(u,v)$ is a ground state solution of \eqref{auto-2}. As in \cite[Lemma 3.5]{ACR-1}, then $u,v>0$ or
$u,v<0$. By Lemma \ref{lem-3.2}, $z_1=(\sqrt{\frac{1-\kappa_0}{1+\beta_0}}w(\sqrt{1-\kappa_0}x),\sqrt{\frac{1-\kappa_0}{1+\beta_0}}w(\sqrt{1-\kappa_0}x))$
is the unique positive solution. From this and the structure of the
system \eqref{auto-2} we infer that the solution of \eqref{auto-2} with both components negative is also unique. Hence
$z_2=(-\sqrt{\frac{1-\kappa_0}{1+\beta_0}}w(\sqrt{1-\kappa_0}x),-\sqrt{\frac{1-\kappa_0}{1+\beta_0}}w(\sqrt{1-\kappa_0}x))$
is the unique negative solution of \eqref{auto-2}. In conclusion,
we have $z=z_1$ or $z=z_2$. As in Lemma \ref{lem-3.5}
and Remark \ref{rem-3.7}, $z_1$ and $z_2$ are
nondegenerate solutions of \eqref{auto-2}. Thus, $\Phi_0(u_n,v_n)=c_0$ by \cite[Lemma
3.7]{ACR-1}, contradicting $(ii)$.
\end{proof}

Henceforth we assume the hypotheses of Lemma \ref{Lemma 7.3}, so that $c_0$ is an isolated critical level of $\Phi_0$. Let $d_0 = \inf\, \{d > c_0 : d \text{ is a critical level of } \Phi_0\}$ and set $\tilde{c}_0 = \min\, \{d_0, 2c_0\}$. Then $\tilde{c}_0 > c_0$ and we have the following lemma.

\begin{lemma} \label{Lemma 7.4}
$\Phi$ satisfies the {\em (PS)}$_d$ condition for all $d \in (c_0,\tilde{c}_0)$.
\end{lemma}

\begin{proof}
Let $d \in (c_0,\tilde{c}_0)$ and let $(u_k,v_k) \in E$ be a (PS)$_d$ sequence for $\Phi$, i.e., $\Phi(u_k,v_k) \to d$ and $\Phi'(u_k,v_k) \to 0$. Since
\[
\Phi(u_k,v_k) - \frac{1}{4}\, \Phi'(u_k,v_k)\, (u_k,v_k) = \frac{1}{4} \left(\|u_k\|^2 + \|v_k\|^2 - 2 \int_{\R^N} (\kappa_0 + \kappa(x))\, u_k v_k \right)
\]
and $\kappa_0 + \sup \kappa(x) < 1$, it follows as in \eqref{b-5} that $(u_k, v_k)$ is bounded. Passing to the renamed subsequence provided by Theorem \ref{profdec} and utilizing Proposition \ref{prop:energy},
\[
\Phi(U^{(1)},V^{(1)}) + \sum_{n \ge 2}\, \Phi_0(U^{(n)},V^{(n)}) = d,
\]
$(U^{(1)},V^{(1)}) \in \mathscr{N} \cup \{(0,0)\}$, and $(U^{(n)},V^{(n)}) \in \mathscr{N}_0 \cup \{(0,0)\}$ for $n \ge 2$. Since $c = c_0$ and $d < 2c_0$, then $(U^{(n)},V^{(n)})$ is nontrivial for at most one $n \ge 1$. Since $d > 0$, then $(U^{(n)},V^{(n)})$ is nontrivial for exactly one $n$, say, for $n_0$. Since $c_0 < d < d_0$ and $\Phi_0$ has no critical levels in this interval, $n_0 = 1$. Then $(u_k,v_k) \to (U^{(1)},V^{(1)})$ as in the proof of Lemma \ref{Lemma 7.2}.
\end{proof}

\begin{lemma} \label{Lemma 7.5}
If
\begin{equation} \label{7.7}
\frac{\left(1 + \frac{|\kappa|_\infty}{1 - \kappa_0}\right)^2}{1 - \max \left\{\frac{|a|_\infty}{a_0},\frac{|b|_\infty}{b_0},\frac{|\beta|_\infty}{\beta_0}\right\}} < \frac{\tilde{c}_0}{c_0},
\end{equation}
then $\Phi(\Gamma(y)) < \tilde{c}_0$ for all $y \in \R^N$.
\end{lemma}

\begin{proof}
Since $\Gamma(y) \in \mathscr{N}$,
\[
\Phi(\Gamma(y)) = \frac{t_y^2}{4} \left[\|u\|^2 + \|v\|^2 - 2 \int_{\R^N} (\kappa_0 + \kappa(x + y))\, uv\, dx\right]
\]
by \eqref{b-5}. Since the numerator in \eqref{7.6} is less than or equal to
\begin{multline*}
\|u\|^2 + \|v\|^2 - 2\, (\kappa_0 - |\kappa|_\infty) \int_{\R^N} uv\, dx = \left(1 + \frac{|\kappa|_\infty}{1 - \kappa_0}\right) \left[\|u\|^2 + \|v\|^2 - 2 \kappa_0 \int_{\R^N} uv\, dx\right]\\[10pt]
- \frac{|\kappa|_\infty}{1 - \kappa_0} \left[\|u\|^2 + \|v\|^2 - 2 \int_{\R^N} uv\, dx\right] \le 4c_0 \left(1 + \frac{|\kappa|_\infty}{1 - \kappa_0}\right)
\end{multline*}
and the denominator is greater than or equal to
\begin{multline*}
\int_{\R^N} \left[(a_0 - |a|_\infty)\, u^4 + (b_0 - |b|_\infty)\, v^4 + 2\, (\beta_0 - |\beta|_\infty)\, u^2 v^2\right] dx\\[10pt]
\ge 4c_0 \left(1 - \max \left\{\frac{|a|_\infty}{a_0},\frac{|b|_\infty}{b_0},\frac{|\beta|_\infty}{\beta_0}\right\}\right),
\end{multline*}
the conclusion follows from \eqref{7.7}.
\end{proof}

The main result of this section is the following theorem.

\begin{theorem}\label{zhang}
Let $a_0, b_0 > 0$, $0 < \kappa_0 < 1$, and assume that \eqref{7.10} and the hypotheses of Lemma \ref{Lemma 7.3} are satisfied. If $\kappa(x), a(x), b(x), \beta(x) \le 0$, with at least one of the inequalities strict on a set of positive measure, and \eqref{7.7} holds, then the system \eqref{7.2} has a nontrivial bound state solution.
\end{theorem}

\begin{proof}
By Lemmas \ref{Lemma 7.1} and \ref{Lemma 7.2}, and \eqref{7.8}, there exists a constant $R > 0$ such that
\begin{equation} \label{7.12}
c_0 < \max_{|y| = R}\, \Phi(\Gamma(y)) < \tilde{c}.
\end{equation}
Let
\[
Q = \Gamma(\overline{B_R(0)}), \qquad S = \{(u,v) \in \mathscr{N} : \xi(u,v) = 0\},
\]
and note that $\partial Q \cap S = \emptyset$ by \eqref{7.11}. We claim that $\partial Q$ links $S$, i.e., $h(Q) \cap S \ne \emptyset$ for every map $h$ in the family
\[
{\mathcal H} = \{h \in C(Q,\mathscr{N}) : h|_{\partial Q} = \id\}.
\]
To see this, let $h \in \mathcal H$ and consider the continuous map $\varphi = \xi \circ h \circ \Gamma : \overline{B_R(0)} \to \R^N$. If $|y| = R$, then $\Gamma(y) \in \partial Q$ and hence $h(\Gamma(y)) = \Gamma(y)$, so $\varphi(y) = y$ by \eqref{7.11}. By the Brouwer fixed point theorem, then there exists $y \in B_R(0)$ such that $\varphi(y) = 0$, i.e., $h(\Gamma(y)) \in S$.

We have
\begin{equation} \label{7.13}
\inf_S\, \Phi = \tilde{c} > \max_{\partial Q}\, \Phi > c_0
\end{equation}
by \eqref{7.12}. Set
\[
d = \inf_{h \in \mathcal H}\, \max_{(u,v) \in h(Q)}\, \Phi(u,v).
\]
Since $\partial Q$ links $S$, $d \ge \tilde{c} > c_0$, and since $\id \in \mathcal H$, $d < \tilde{c}_0$ by Lemma \ref{Lemma 7.5}, so $\Phi$ satisfies the (PS)$_d$ condition by Lemma \ref{Lemma 7.4}. So $d$ is a critical value of $\Phi$ by a standard argument.
\end{proof}

\begin{proof}[Proof of Theorem \ref{th1.2}]
From Theorem \ref{zhang}, we know that the results of Theorem
\ref{th1.2} hold.
\end{proof}

\textbf{Acknowledgements:} This work was completed while K. Perera was visiting the Academy of Mathematics and Systems Science at the Chinese Academy of Sciences, and he is grateful for the kind hospitality of the host institution. J. Wang was supported by Natural Science Foundation of China (No: 11571140) and Natural Science Foundation of Jiangsu Province (Nos: BK2012282, BK20150478). Z.-T. Zhang was supported by National Natural Science Foundation of China (Nos: 11325107, 11271353, 11331010).

\bibliographystyle{plain}

\end{document}